\documentclass[11pt]{article}

\usepackage{fullpage}
\usepackage{amsmath,amsthm,amsfonts,dsfont,mathrsfs}
\usepackage{amssymb,latexsym,graphicx}
\usepackage{lmodern}
\usepackage{mathtools}
\usepackage{hyperref}
\usepackage[hyperpageref]{backref}
\usepackage{xcolor}
\hypersetup{
	colorlinks,
	linkcolor={red!75!black},
	citecolor={blue!75!black},
	urlcolor={blue!75!black}
}
\usepackage{enumerate}
\usepackage{enumitem}
\usepackage{boxedminipage}
\usepackage{caption}
\usepackage{subcaption}

\newtheorem{theorem}{Theorem}[section]

\newtheorem*{remark}{Remark}
\newtheorem{proposition}[theorem]{Proposition}
\newtheorem{lemma}[theorem]{Lemma}

\newtheorem{claim}[theorem]{Claim}
\newtheorem{fact}[theorem]{Fact}

\newtheorem{corollary}[theorem]{Corollary}

\newcommand{\N}{\ensuremath{\mathbb{N}}}

\newcommand{\R}{\ensuremath{\mathbb{R}}}
\newcommand{\Z}{\ensuremath{\mathbb{Z}}}

\newcommand{\Tr}{{\rm Tr}}

\renewcommand{\epsilon}{\varepsilon} 

\renewcommand{\vec}[1]{\ensuremath{\boldsymbol{#1}}}
\newcommand{\basis}{\ensuremath{\mathbf{B}}}

\DeclareMathOperator*{\argmin}{argmin}
\DeclareMathOperator*{\expect}{\mathbb{E}}
\newcommand{\grad}{\nabla}

\newcommand{\M}{\mathcal{M}}

\newcommand{\V}{\mathcal{V}}

\newcommand{\mubar}{\overline{\mu}}

\newcommand{\KLCstable}{C_{\mu}(n)}

\newcommand{\lat}{\mathcal{L}}
\DeclareMathOperator{\vol}{vol}
\DeclareMathOperator{\rank}{rank}
\DeclareMathOperator{\dist}{dist}
\DeclareMathOperator{\spn}{span}

\DeclarePairedDelimiter\inner{\langle}{\rangle}

\DeclarePairedDelimiter\floor{\lfloor}{\rfloor}
\DeclarePairedDelimiter\ceil{\lceil}{\rceil}

\begin{document}

\title{\bf A Reverse Minkowski Theorem}
\author{
Oded Regev\thanks{Courant Institute of Mathematical Sciences, New York
 University.}%
~\thanks{Supported by the Simons Collaboration on Algorithms and Geometry and by the National Science Foundation (NSF) under Grant No.~CCF-1320188. Any opinions, findings, and conclusions or recommendations expressed in this material are those of the authors and do not necessarily reflect the views of the NSF.}\\
\and
Noah Stephens-Davidowitz\thanks{Cornell University.}%
~\thanks{ Much of this work was done while at New York University, supported by the National Science Foundation (NSF) under Grant No.~CCF-1320188. Part of this work was done while visiting Chris Peikert at the University of Michigan.}\\
\texttt{noahsd@gmail.com}
}
\date{}
\maketitle

\begin{abstract}
We prove a conjecture due to Dadush, showing that if $\lat \subset \R^n$ is a lattice such that $\det(\lat') \ge 1$ for all sublattices $\lat' \subseteq \lat$, then 
\[
\sum_{\vec y \in \lat} e^{-\pi t^2 \|\vec y\|^2} \le 3/2 \; ,
\]
where $t := 10(\log n + 2)$. 
From this we derive bounds on the number of short lattice vectors, which can be viewed as a partial converse to Minkowski's celebrated first theorem. We also derive a bound on the covering radius.
\end{abstract}

\section{Introduction}
\label{sec:intro}
 A lattice $\lat \subset \R^n$ is the set of integer linear combinations of linearly independent basis vectors $\basis = (\vec{b}_1,\ldots, \vec{b}_n)$. The determinant of the lattice, $\det(\lat) = |\det(\basis)|$, is a measure of its global density in the sense that
\[
\det(\lat) = \lim_{r \rightarrow \infty} \frac{\vol(r B_2^n)}{|\lat \cap r B_2^n|}
\; ,
\]
where $rB_2^n$ denotes the closed Euclidean ball of radius $r >0$, whose volume is $(\pi n)^{-1/2} (2\pi e r^2/n)^{n/2}(1+o(1))$. (Here and elsewhere, we write $o(1)$ for an arbitrary function that approaches zero as the dimension $n$ approaches infinity.)

Minkowski's celebrated first theorem shows that a lattice with small determinant must have short non-zero vectors~\cite{MinkowskiBook}. This is one of the foundational results in the study of lattices and the geometry of numbers, and it has innumerable applications. We consider the following point-counting form of this theorem due to Blichfeldt and van der Corput,\footnote{They actually showed the slightly stronger bound $|\lat \cap rB_2^n| \geq 2\floor{2^{-n} \cdot \vol(r B_2^n)}+1$ and considered arbitrary norms, not just $\ell_2$.  (See, e.g.,~\cite[Thm.~1 of Ch.~2, Sec.~7]{Lekkerkerker_book}.)} 
which says that a lattice with small determinant must have \emph{many} short points, or informally, that ``global density implies local density.''

\begin{theorem}[\cite{vanderCorput1936}]
	\label{thm:minkowski}
For any lattice $\lat \subset \R^n$ with $\det(\lat) \leq 1$ and $r > 0$,
\begin{equation*}
|\lat \cap r B_2^n| 
\geq 2^{-n} \cdot \vol(r B_2^n) =  \frac{1}{\sqrt{\pi n}} \Big(\frac{\pi e r^2}{2 n}  \Big)^{n/2} (1+o(1))
\; . 
\end{equation*}
\end{theorem}

It is quite natural to ask whether a converse of Theorem~\ref{thm:minkowski} holds. In particular, if a lattice has sufficiently many short points, does it necessarily have small determinant? Does local density imply global density?

It is easy to see that the answer is actually no. Consider, for example, the lattice generated by the vectors $(1/t,0)$ and $(0,t^2)$ for some arbitrarily large $t$. This lattice has at least $2\floor{tr} + 1$ points of norm at most $r$, but it has arbitrarily large determinant $t$. Notice, however, that this lattice contains a \emph{sublattice} generated by $(1/t,0)$ that does have small determinant. This leads us to a more refined question: 
\begin{quote}
If a lattice has sufficiently many short points, does it necessarily have a small-determinant \emph{sublattice}? Does local density imply global density  restricted to a subspace?
\end{quote}
Equivalently, in the contrapositive, the question asks for an upper bound on the number of lattice points in a ball given that there is no sublattice of small determinant. 

Dadush conjectured a suitably precise answer to these questions~\cite{priv:Daniel}.
He later studied this conjecture in depth in joint work with the first named author~\cite{DR16}. 
Among other things, they showed a number of applications of the conjecture 
(from computational complexity of lattice problems to Brownian motion on flat tori)
and gave some evidence for it.
We refer the reader to~\cite{DR16} for a full list of their results.

Our main result is a proof of the conjecture of Dadush, which in particular implies the applications mentioned above.  

\begin{theorem}[Reverse Minkowski Theorem]
\label{thm:RM}
For any lattice $\lat \subset \R^n$ with $\det(\lat') \geq 1$ for all sublattices $\lat' \subseteq \lat$, 
\[
\rho_{1/t}(\lat) \leq \frac{3}{2}
\; ,
\]
where $t := 10(\log n + 2)$.
\end{theorem}

Here, for a lattice $\lat \subset \R^n$ and $s>0$, 
\[
\rho_s(\lat) := \sum_{\vec{y} \in \lat} e^{-\pi \|\vec{y}\|^2/s^2}
\; 
\]
 is the \emph{Gaussian mass} of the lattice with \emph{parameter} $s$.
This can be seen as a smooth version of the point-counting function $r \mapsto |\lat \cap r B_2^n|$, with the parameter
$s$ playing the role of the radius $r$, and it arises naturally in a number of 
contexts (often in the form of the theta function, $\Theta_{\lat}(iy) := \rho_{1/\sqrt{y}}(\lat)$). 
In particular, Theorem~\ref{thm:RM} immediately implies that $|\lat \cap r B_2^n| \leq 3e^{\pi t^2 r^2}/2$ for any radius $r > 0$. (We note that the constant $3/2$ in this bound and the theorem statement is chosen for convenience, and a similar statement holds with any constant strictly bigger than $1$.)

One can view Theorem~\ref{thm:RM} as relating the parameters
\[
\eta^*(\lat) := \inf\{ t \ : \ \rho_{1/t}(\lat) \leq 3/2 \}
\; 
\] 
(known as the smoothing parameter of the dual lattice~\cite{MR04}) and 
\[
\eta_{\det}(\lat) := \max_{\lat' \subseteq \lat} \det(\lat')^{-1/\rank(\lat')}
\; .
\]
Specifically, we claim that 
\begin{equation}
\label{eq:RM_eta}
\frac{2}{3} \cdot \eta_{\det}(\lat) \leq \eta^*(\lat) \leq 10(\log n + 2) \cdot \eta_{\det}(\lat)
\;.
\end{equation}
Indeed, the first inequality is an immediate consequence of the Poisson Summation Formula, Eq.~\eqref{eq:PSF} (see Eq.~\eqref{eq:lowerboundonrho}). When $\eta_{\det}(\lat)=1$, the second inequality is precisely Theorem~\ref{thm:RM}; the general case follows by noting that both $\eta_{\det}(\lat)$ and $\eta^*(\lat)$ behave identically under scaling of $\lat$ (homogeneous of degree $-1$). 
Eq.~\eqref{eq:RM_eta} is not far from tight, as can be seen by noting that $\eta^*(\Z^n) = \sqrt{\log n/\pi} + o(1)$ and $\eta_{\det}(\Z^n) = 1$.

In Section~\ref{sec:count}, we extend Theorem~\ref{thm:RM} to obtain a bound on the Gaussian mass for all parameters, as follows.

 \begin{theorem}
 	\label{thm:RM_all_parameters}
 	Let $t := 10(\log n + 2)$. Then, for  any lattice $\lat \subset \R^n$ with $\det(\lat') \geq 1$ for all sublattices $\lat' \subseteq \lat$,
 	\begin{enumerate}
 		\item \label{item:RM} $\rho_{s}(\lat) \leq 1 + e^{-\pi(1/s^2 - t^2)}/2$ for any $s \leq 1/t$;
 		\item \label{item:RM_convex} $\rho_s(\lat) \leq (Cst)^{n/2}$ for any $1/t < s < t$ and some universal constant $C > 1$; and
 		\item \label{item:RM_dual} $\rho_s(\lat) \leq 2s^n$ for any $s \geq t$.
 	\end{enumerate}
\end{theorem}

\noindent Theorem~\ref{thm:RM_all_parameters} implies the following point-counting bounds. (See Section~\ref{sec:count} for the proof.) 

\begin{corollary}
	\label{cor:counting}
	Let $t := 10(\log n + 2)$. Then, for any lattice $\lat \subset \R^n$ with $\det(\lat') \geq 1$ for all sublattices $\lat' \subseteq \lat$, 
	and every shift vector $\vec{u} \in \R^n$,
	\begin{enumerate}
		\item \label{item:RM_bound} for any $r \geq 1$, $|\lat \cap (r B_2^n + \vec{u})| \leq 3e^{\pi t^2 r^2}/2$;
		\item \label{item:convex_bound} 
		for any $\sqrt{n/(2\pi)} \cdot  t^{-1} \leq r \leq \sqrt{n/(2\pi)} \cdot t $, $|\lat \cap (r B_2^n + \vec{u})|  \leq (C tr/\sqrt{n})^{n/2}$ for some universal constant $C > 0$; and 
		\item \label{item:dual_packing_bound} for any $r \geq \sqrt{n/(2\pi)} \cdot t$, $|\lat \cap (r B_2^n + \vec{u})|  \leq 2(2\pi e r^2/n)^{n/2}$.
	\end{enumerate}
\end{corollary}

\noindent In Section~\ref{sec:examples}, we discuss the tightness of Theorem~\ref{thm:RM_all_parameters} and Corollary~\ref{cor:counting}.

\subsection{Approximation to the covering radius}
\label{sec:KL_intro}

The \emph{covering radius} $\mu(\lat)$ of a lattice $\lat \subset \R^n$ is the maximal distance from any point in $\R^n$ to the lattice, or equivalently, the minimum radius $r$ such that $\lat + r B_2^n = \R^n$. 
It follows from the definition that $\mu(\lat)$ must be at least the radius of a ball of volume $\det(\lat)$, which is at least $\sqrt{n/(2 \pi e)} \det(\lat)^{1/n}$. 
By considering projections, Kannan and Lov{\'a}sz~\cite{KL88} improved this lower bound, as follows. 
Let $\pi_{W^\perp}(\lat)$ be the projection of the lattice onto the space $W^\perp$ orthogonal to some \emph{lattice subspace} $W \subset \R^n$---a subspace spanned by $k < n$ linearly independent lattice vectors.\footnote{The projection $\pi_{W^\perp}(\lat)$ is a lattice if and only if $W$ is a lattice subspace.} 
Then clearly $\mu(\lat) \geq \mu(\pi_{W^\perp}(\lat))$, and the latter is at least $(\dim(W^\perp)/(2 \pi e))^{1/2} \cdot \det(\pi_{W^\perp}(\lat))^{1/\dim(W^\perp)}$.
So, we obtain the lower bound
\[
\mu(\lat) 
\ge \frac{1}{\sqrt{2\pi e}} \cdot \mu_{\det}(\lat) 
\; ,
\]
where
\begin{align*}
\mu_{\det}(\lat) &:= \max_{W \subset \R^n} \sqrt{\dim(W^\perp)} \cdot \det(\pi_{W^\perp}(\lat))^{\frac{1}{\dim(W^\perp)}} \; ,
\end{align*}
with the maximum taken over lattice subspaces $W \subset \R^n$.
To make the analogy to $\eta_{\det}$ more apparent, we apply duality and write
\begin{align*}
\mu_{\det}(\lat) = \max_{W \subset \R^n} \sqrt{\dim(W^\perp)} \cdot \det(\pi_{W^\perp}(\lat)^*)^{-\frac{1}{\dim(W^\perp)}} =  \max_{\mathcal{M} \subseteq \lat^*} \sqrt{\rank(\mathcal{M})} \cdot \det(\mathcal{M})^{-\frac{1}{\rank(\mathcal{M})}}
\; ,
\end{align*}
where $\lat^*$ denotes the dual of a lattice $\lat$, the first equality uses the fact that the determinant of the dual lattice is the reciprocal of the determinant of the lattice, and the second equality uses the fact that  $\pi_{W^\perp}(\lat)^* = W^\perp \cap \lat^*$.
Kannan and Lov{\'a}sz also observed the upper bound
\[
\mu(\lat) \leq C \sqrt{n} \cdot \mu_{\det}(\lat)
\; 
\]
(see~\cite[Theorem 11.1]{DR16} for a proof),
and asked whether a better upper bound could be found.%
\footnote{They also proved similar bounds for arbitrary norms~\cite[Corollary 3.11]{KL88}.} 
In Section~\ref{sec:KL}, we use Theorem~\ref{thm:RM} to derive the following improved bound.

\begin{theorem}[Covering-radius approximation]
\label{thm:KL}
For any lattice $\lat \subset \R^n$, 
\begin{equation}
\label{eq:covradchar}
\frac{1}{\sqrt{2\pi e}} \cdot \mu_{\det}(\lat) \leq \mu(\lat) \leq 10(\log n + 10)^{3/2} \cdot \mu_{\det}(\lat)
\; .
\end{equation}
\end{theorem}

We emphasize that Dadush and Regev~\cite{DR16} already proved that Theorem~\ref{thm:KL} (with slightly weaker parameters) 
would follow from a proof of Theorem~\ref{thm:RM}. 
Although our proof is shorter and achieves slightly better 
parameters, it is conceptually similar to the one in~\cite{DR16}.

We note that the specific polylogarithmic factor that we obtain is likely not optimal. 
In fact, in Theorem~\ref{thm:KL_slicing} we prove a bound similar to that in Eq.~\eqref{eq:covradchar}
that replaces the factor $10(\log n + 10)^{3/2}$ by $C \sqrt{\log n}$, assuming the celebrated Slicing Conjecture~\cite{Bourgain91-slicing,Klartag06,Chen21,KlartagLehec2022}. However,
it is not difficult to show that this factor cannot be smaller than $\sqrt{\log n/(4e)} + o(1)$.\footnote{%
	\label{foot:harmonic_cov_rad}%
	Consider the lattice $\lat$ generated by $(\vec{e}_1, \vec{e}_2/2, 2\vec{e}_3/3^{3/2}, \ldots, (n-1)^{(n-1)/2}\vec{e}_n/n^{n/2})$. It is not difficult to verify that $\mu_{\det}(\lat) = 1$, but 
	\begin{align*}
		\mu(\lat)^2 = 1/4 + \sum_{k=2}^n \frac{(k-1)^{k-1}}{4k^k} 
		= 1/4 + \sum_{k=2}^n \frac{(1-1/k)^{k}}{4(k-1)} 
		= \sum_{k=2}^n \frac{1}{4e(k-1)} + O(1) 
		= \frac{\log n}{4e} + O(1)
		\; .
	\end{align*}
	Therefore, $\mu(\lat) = \sqrt{\log n/(4e)} + o(1)$.
}

\paragraph{Covering radius of stable lattices and Minkowski's Conjecture.} 
We say that a lattice $\lat \subset \R^n$ is \emph{stable} if $\det(\lat) = 1$ and $\det(\lat') \geq 1$ for all sublattices $\lat' \subseteq \lat$. Stable lattices arise in a number of contexts~\cite{HN75,Stuhler76,Grayson84} and they play an important role in the sequel. Shapira and Weiss showed that a tight bound of $\mu(\lat) \leq \mu(\Z^n) = \sqrt{n}/2$ on the covering radius of stable lattices would imply a well-known conjecture attributed to Minkowski~\cite{SW16}. Specifically, the conjecture asserts that for every lattice $\lat \subset \R^n$ with $\det(\lat) = 1$ and vector $\vec{t} = (t_1,\ldots, t_n) \in \R^n$,
\begin{equation}
    \label{eq:minkowski_conjecture}
    \inf_{\vec{y} \in \lat} \prod_i |y_i-t_i| \leq 2^{-n}
    \; .
\end{equation}
(See~\cite{SW16} and~\cite{Solan16}.) 

We do not manage to prove a bound tight enough to imply Eq.~\eqref{eq:minkowski_conjecture}, but en route to proving Theorem~\ref{thm:KL} we do show that $\mu(\lat) \leq 4 \sqrt{n} (\log n + 10)$ for all stable lattices. (See Theorem~\ref{thm:KL_stable}.) We also observe that a very strong resolution to the Slicing Conjecture would yield the desired tight bound, when combined with a recent result due to Magazinov~\cite{Magazinov}. (See Theorem~\ref{thm:KL_stable_slicing} and the discussion afterwards.)

\subsection{An optimal bound on the Gaussian mass for ``extreme'' parameters} 

It is tempting to ask whether $\rho_s(\lat) \leq \rho_s(\Z^n)$ for any lattice $\lat \subset \R^n$ such that $\det(\lat') \geq 1$ for all sublattices $\lat' \subseteq \lat$ and any parameter $s > 0$. 
(See Section~\ref{sec:open_problems}.) The next theorem shows that indeed $\rho_s(\lat) \leq \rho_s(\Z^n)$ for such lattices, but only for ``extremely low'' or ``extremely high'' parameters $s$. (See Section~\ref{sec:all_params} for the proof.)

 \begin{theorem}
	\label{thm:extreme_parameters}
	For any lattice $\lat \subset \R^n$ such that $\det(\lat') \geq 1$ for all sublattices $\lat' \subseteq \lat$ and parameter $s >0$ such that either $s \leq \sqrt{2\pi/(n+2)}$ or $s \geq \sqrt{(n+2)/(2\pi)}$, we have $\rho_{s}(\lat) \leq \rho_{s}(\Z^n)$.
\end{theorem}

\noindent We hope that the proof of Theorem~\ref{thm:extreme_parameters} might provide some hints as to how to extend
it to all parameters $s$.

\subsection{Proof overview}
\label{sec:overview}

In this section, we give a high-level overview of the proof of Theorem~\ref{thm:RM}.

\paragraph{Bounding the mass of stable lattices. } 
Recall that a lattice $\lat$ is \emph{stable} if $\det(\lat) = 1$ and $\det(\lat') \geq 1$ for all sublattices $\lat' \subseteq \lat$. I.e., stable lattices are determinant-one lattices that satisfy the assumption in Theorem~\ref{thm:RM}. 
In this proof overview, we focus on bounding the Gaussian mass $\rho_s(\lat)$ of stable lattices $\lat$.
As it turns out, the general case then follows easily. 

Crucially, the stable lattices form a compact subset of the set of determinant-one lattices, so that the continuous function $\rho_s(\lat)$ must attain a global maximum over the set of stable lattices. We may therefore restrict our attention to a lattice that corresponds to this global maximum.
If this lattice is on the \emph{boundary} of the set of stable lattices, then it has a strict sublattice $\lat'$ with determinant one. We can then ``split the lattice'' at $\lat'$. Namely, we can replace the original lattice $\lat$  by the direct sum $\lat' \oplus \lat/\lat'$, where both $\lat'$ and $\lat/\lat'$ are stable. By using the fact that the Gaussian has a positive Fourier transform, it is not difficult to prove that
\[
\rho_s(\lat) \leq \rho_s(\lat' \oplus \lat/\lat') = \rho_s(\lat')\rho_s(\lat/\lat')
\; .
\]
(See Lemma~\ref{lem:direct_sum_rho}.) So, we have reduced the question to a lower-dimensional one. 
Therefore, if we could show that for any dimension, the global maximizer is on the boundary, then we could use induction to show that the global maximizer of the Gaussian mass is simply the integer lattice $\Z^n = \Z \oplus \cdots \oplus \Z$. 

Indeed, this is how we prove Theorem~\ref{thm:extreme_parameters} (in Section~\ref{sec:all_params}), which shows that $\Z^n$ has maximal Gaussian mass for certain ``extreme'' parameters $s$. 
For such parameters, by taking the second derivative, we show that a stable lattice cannot be a local maximum over the  set of determinant-one lattices.
Therefore, the global maximizer of $\rho_s(\lat)$ over the compact subset of stable lattices must be on the boundary, and we can perform the ``splitting'' procedure described above to show by induction that  $\rho_s(\lat) \leq \rho_s(\Z^n)$.

However, it was recently shown by Heimendahl et al.\ that stable local maxima do exist for some parameters $s$~\cite{HeimendahlMTVZ21}.\footnote{This was originally posed as an open  question in an earlier version of this paper.} As a potential way around this issue, we could use a natural and very elegant idea due to Shapira and Weiss~\cite{SW16}---we could try to directly bound the value of $\rho_s(\lat)$ at any local maximum. Then, either the global maximum of $\rho_s(\lat)$ over the set of stable lattices is one of these local maxima, in which case we can apply this bound; or it is on the boundary, in which case we can ``split the lattice'' as above. (Shapira and Weiss suggested using this approach to bound the \emph{covering radius} of stable lattices, which is also known to have local maxima~\cite{DSV12}; they showed that a tight bound would resolve Minkowski's Conjecture~\cite{SW16}.)

\paragraph{Enter the Voronoi cell. } Unfortunately, directly bounding the value of $\rho_s(\lat)$ at local maxima seems to be beyond our grasp. So, instead of working with $\rho_s(\lat)$ directly, we work with a proxy for it: the Gaussian mass of the Voronoi cell of the lattice
\[
\gamma_s(\V(\lat)) := \int_{\V(\lat)/s} e^{-\pi \|\vec{x}\|^2}{\rm d}\vec{x}
\; ,
\]
where the Voronoi cell is the set of all points that are at least as close to the origin as to any other lattice point,
\[
\V(\lat) := \{\vec{x} \in \R^n \ : \ \forall \vec{y} \in \lat,\ \|\vec{x}\| \leq \|\vec{y} - \vec{x}\|\}
\; .
\]
An elegant proof due to Chung, Dadush, Liu, and Peikert~\cite{CDLP12} shows that $\rho_s(\lat)$ 
is at most $1/\gamma_s(\V(\lat))$. (See Lemma~\ref{lem:rho_gamma}.)
So, in order to prove an upper bound on $\rho_s(\lat)$, it suffices to prove a lower 
bound on $\gamma_s(\V(\lat))$. 

We accomplish this via the approach described above. 
Namely, we reduce the problem to bounding the value of
$\gamma_s(\V(\lat))$ at local minima $\lat$. (We do not know whether these local minima exist.) 
For such an $\lat$, we consider two functions, both defined over the set of all determinant-one matrices $A \in \mathrm{SL}_n(\R)$: 
$g(A) = \gamma_s(\V(A\lat))$ and 
$h(A)=\gamma_s(A \V(\lat))$.
Notice that the value we wish to bound is $g(I_n)=h(I_n)$. Moreover, 
as we show (in Section~\ref{sec:gradient}),
the two functions have the same gradient at $A=I_n$ and therefore the fact that $g$ has a local minimum at $A=I_n$ implies that $h$ has a critical point there. 
Using a result due to Bobkov~\cite{bobkov11}, which itself follows from a deep theorem due to Cordero-Erausquin, Fradelizi, and Maurey~\cite{CFM04},\footnote{We 
		note in passing that one can prove Theorem~\ref{thm:RM} (at least up to constants) 
		without using this rather heavy hammer by considering local maxima of the 
		\emph{$\ell$-norm} of the Voronoi cell instead of local minima of the Gaussian 
		mass of the Voronoi cell.} 
we can show that any such critical point of $h$ must actually 
be a global \emph{maximum}. I.e., in the language of convex geometry, 
the Voronoi cell is in a position that maximizes the Gaussian mass. 
(Note the rather surprising jump from a presumed local minimum over the 
set of determinant-one lattices to a global \emph{maximum} over the set 
of positions of the Voronoi cell.) 
Finally, we complete the proof by 
applying the celebrated $\ell\ell^*$ theorem~\cite{FigielT79,Lewis79,Pisier82}, which implies that for $s=1/t$,
the global maximum of $h$ is at least $2/3$,
where $t := 10(\log n + 2)$ as in Theorem~\ref{thm:RM}.

\subsection{Related work}

Our main theorem was originally conjectured by Dadush~\cite{priv:Daniel}. 
Dadush together with the first named author described several applications of the conjecture~\cite{DR16}. 
In particular, they showed the connection between this conjecture and the Kannan-Lov{\'a}sz-style covering-radius approximation given in Theorem~\ref{thm:KL}. 
They also used a result from convex geometry (specifically the Milman-Pisier Theorem~\cite{MP87}) as evidence 
for the conjecture. That theorem is related to the $\ell \ell^*$ theorem that we use
in our proof.

The high-level outline of our proof (in which we obtain a bound on a lattice parameter by reducing the question to stable local extrema) is due to Shapira and Weiss~\cite{SW16}. They showed that an important conjecture attributed to Minkowski would follow if we could prove that $\Z^n$ has maximal covering radius amongst all stable lattices
(i.e., that the covering radius of an $n$-dimensional stable lattice is at most $\sqrt{n}/2$). 
They then observed that it would suffice to bound the covering radius of the lattices corresponding 
to local maxima of the covering radius function over the set of determinant-one lattices. 

Stable lattices were introduced (in a more general context) by Harder and Narasimhan~\cite{HN75} and by Stuhler~\cite{Stuhler76}. Our presentation more-or-less follows that of Grayson~\cite{Grayson84}.

Counting the number of lattice points in a ball is a classical question, and a summary of all that is known is far beyond the scope of this paper. (See, e.g.,~\cite{ConwaySloaneBook98}.) In particular, much research has gone into studying the relationship between the number of points in a ball of radius $r$ and the determinant of the densest \emph{one-dimensional} sublattice, written $\lambda_1(\lat)$. (I.e., $\lambda_1(\lat)$ is the length of the shortest non-zero vector in the lattice.) It is easy to see that the number of lattice points in a ball of radius $r \geq \lambda_1(\lat)$ is at most $(Cr/\lambda_1(\lat))^n$, which is essentially the best possible bound based on $\lambda_1(\lat)$.\footnote{Finding the exact best possible bounds on $|\lat \cap s\lambda_1(\lat) B_2^n|$ in various regimes is a fascinating classical problem that is still an active area of research. For example, when $s = 1$, this is known as the lattice ``kissing number'' problem, and the limit as $s \to \infty$ is the lattice sphere-packing problem. See, e.g.,~\cite{KL78,ConwaySloaneBook98,CK09}.} We consider the densest sublattice of any dimension (not just the densest one-dimensional sublattice) to obtain bounds that are much stronger in many cases. (Other authors have considered other generalizations of $\lambda_1(\lat)$ to derive incomparable bounds. E.g.,~\cite{Henk02}.)

Many authors have considered the extrema of various lattice parameters over the set of determinant-one lattices. Voronoi famously characterized the local maxima of the length of the shortest non-zero vector~\cite{Voronoi1908}, and a long line of work has gone into finding the specific global maxima in various dimensions. (See, e.g.,~\cite{ConwaySloaneBook98,CK09}.) Similarly, Montgomery~\cite{Montgomery} and Sarnak and Str{\"o}mbergsson~\cite{SS06} considered the minima of the Gaussian mass $\rho_s(\lat)$ and closely related functions.

Informally, the results mentioned above (and almost all literature on this topic since Minkowski) were concerned with the ``best'' lattices. E.g., they primarily considered lattices with the largest minimum distance, the smallest covering radius, the minimal Gaussian mass, etc. We are in some sense interested in the ``worst'' lattices. Thus, we consider \emph{maxima} of the Gaussian mass, maxima of the covering radius (as in~\cite{DSV12}), etc. (These questions only make sense over a bounded subset of the determinant-one lattices, such as the stable lattices.) 
Note that, while the ``best'' lattices tend to have fascinating properties (see, e.g.,~\cite{ConwaySloaneBook98}), in our setting the ``worst'' lattice that we know of is $\Z^n$.

We also note two follow-up works. First, Lovett and the first named author used Theorem~\ref{thm:RM} to give a counterexample to a very strong variant of the polynomial Freiman-Ruzsa conjecture over the integers~\cite{LR16}. This variant was introduced by Green (who suggested that it was likely to be false)~\cite{Green07}. Second, Dadush showed a number of applications of Theorems~\ref{thm:RM} and~\ref{thm:KL}, including an algorithm for finding dense lattice subspaces and a remarkably tight approximation to the covering radius in terms of the so-called canonical filtration~\cite{DadApproximatingCovering19} (as defined in Section~\ref{sec:stable}). 

The reader might also be interested in the lecture notes of Bost providing a broader perspective on the results of this paper~\cite{Bost20}.

\subsection{Directions for future work}
\label{sec:open_problems}

The most obvious direction for future work is to try to obtain a better value for $t$ in Theorem~\ref{thm:RM}. As far as we know, the correct value could be as small as $t = \eta^*(\Z^n) = \sqrt{\log (n)/\pi} + o(1)$.
Our proof seems to be loose in two places: (1) Theorem~\ref{thm:ell_position_mass}, which bounds the maximal Gaussian mass of convex bodies; and (2) the induction argument in the proof of Proposition~\ref{prop:RM_stable}. It seems that one would need to improve both parts of the proof to obtain a significantly better bound.

A more ambitious goal would be to prove that 
\begin{equation}
\label{eq:tight_RM}
\rho_s(\lat) \leq \rho_s(\Z^n)
\end{equation}
for all $s >0$ and all lattices $\lat \subset \R^n$ such that $\det(\lat') \geq 1$ for all sublattices $\lat' \subseteq \lat$. 
To that end, Eisenberg and the authors recently proved that this is true in an average sense over $s$, which can be stated as an inequality relating Epstein zeta functions as follows:
\[
    \sum_{\vec{y} \in \lat \setminus \{\vec0\}} \|\vec{y}\|^{-2\sigma} \leq \sum_{\vec{z} \in \Z^{n} \setminus \{\vec0\}} \|\vec{z}\|^{-2\sigma}
\]
for all such lattices $\lat \subset \R^n$ and all $\sigma > n/2$~\cite{EisenbergRS22}. They  prove this by showing that the Epstein zeta function has no local maxima over the set of determinant-one lattices (and over certain subsets of this set). One might think to try to prove something similar for $\rho_s(\lat)$.
However, local maxima are known to exist for some parameters $s$~\cite{HeimendahlMTVZ21}. 
As an alternative, one can try using the technique of ``characterizing the local extrema'' that we use to prove Theorem~\ref{thm:RM}. For this, we note that any local maximum of $\rho_s(\lat)$ must correspond to an ``isotropic'' lattice $\lat$ in the sense that 
\[
\sum_{\vec{y} \in \lat}  \vec{y}\vec{y}^T e^{-\pi \|\vec{y}\|^2/s^2} = \alpha \cdot I_n
\]
for some scalar $\alpha > 0$. So, it would suffice to show Eq.~\eqref{eq:tight_RM} for (stable) ``isotropic'' lattices. Unfortunately, we do not know how to make use of this.

Recall from Eq.~\eqref{eq:RM_eta} that Theorem~\ref{thm:RM} gives quite a tight approximation to the smoothing parameter $\eta^*(\lat)$. However, an analogous tightness result does not hold for Theorem~\ref{thm:RM_all_parameters} and Corollary~\ref{cor:counting}. 
Dadush and Regev therefore suggested a potential refinement that depends on ``the full spectrum of dense sublattices,'' $\min_{\lat' \subseteq \lat,\ \rank(\lat') = k} \det(\lat')^{1/k}$ for $k = 1,\ldots, n$, rather than just $\min_{\lat' \subseteq \lat} \det(\lat')^{1/\rank(\lat')}$~\cite[Section 9]{DR16}. This could potentially give a tight characterization of $|\lat \cap rB_2^n|$ for all radii $r$ and all lattices $\lat \subset \R^n$.

One can also consider generalizations of Theorems~\ref{thm:KL} and~\ref{thm:RM} to arbitrary norms, as discussed in~\cite{KL88} and~\cite[Section 9]{DR16} respectively. Extending Theorem~\ref{thm:KL} to arbitrary norms could potentially yield faster algorithms for Integer Programming~\cite{thesis/D12}. Unfortunately, a natural generalization of Theorem~\ref{thm:RM} actually fails. (See~\cite[Section 9]{DR16}.)

\subsection*{Acknowledgments}

We are extremely grateful to Daniel Dadush for sharing his conjecture with us and for many helpful discussions. 
We are also indebted to Barak Weiss for introducing us to stable lattices and their salient properties and for telling us about the technique of ``handling the stable local maxima separately'' from his paper with Uri Shapira~\cite{SW16}. We also thank him for his help with the proof of Proposition~\ref{prop:KL_stable_to_unstable}.
We thank Bo'az Klartag for referring us to~\cite{bobkov11} and for other
useful comments. We thank Ronen Eldan for showing us how to greatly simplify the proof of Theorem~\ref{thm:same_grad_VL}.
The first named author thanks Prof.\ Volker Kempe for his help with the two rascals.

\section{Preliminaries}
\label{sec:prelims}
We use $c, C, C'$ to denote arbitrary positive universal constants, whose value might change from one occurrence to the next. Logarithms are base $e$ unless otherwise specified. Vectors $\vec{x} \in \R^n$ are column vectors. We write $\|\vec{x}\| $ to represent the Euclidean norm of $\vec{x}$, and we write $I_n$ for the identity matrix in $n$ dimensions. For a matrix $A \in \R^{n \times n}$, we write $A^T$ for the transpose of $A$. 
We write $B_2^n := \{\vec{x} \in \R^n \ : \ \|\vec{x}\|\leq 1\}$ for the Euclidean ball in $\R^n$. We write $\pi_{S}(\vec{x})$ for the orthogonal projection of $\vec{x}$ onto $\spn(S)$ for some $S \subseteq \R^n$. (E.g., $\pi_{\{\vec{y}\}}(\vec{x}) = \inner{\vec{y}, \vec{x}}\vec{y}/\|\vec{y}\|^2$.) 
For two additive subgroups $S_1 \subseteq \R^n$ and $S_2 \subseteq \R^m$, their direct sum $S_1 \oplus S_2 \subseteq \R^{n + m}$ is $\{(\vec{x},\vec{y})\ :  \vec{x} \in S_1, \vec{y} \in S_2\}$.

A \emph{convex body} $K \subset \R^n$ is a convex compact subset of $\R^n$ with non-empty interior. It is \emph{symmetric} if $-K = K$. A \emph{position} of a convex body is simply $A K$ for a determinant-one matrix $A$.

\subsection{Lattices}
\label{sec:prelims_lattices}

A \emph{lattice} $\lat \subset \R^n$ of rank $d$ is the set of integer linear combinations of linearly independent basis vectors $\basis := (\vec{b}_1,\ldots, \vec{b}_d)$,
\[
\lat = \lat(\basis) := \Big\{ \sum_{i=1}^d a_i \vec{b}_i \ : \ a_i \in \Z \Big\}
\; .
\]
We typically treat lattices as though they are full rank (i.e., $d = n$) by implicitly identifying $\spn(\lat)$ with $\R^d$.
The \emph{dual lattice}
\[
\lat^* := \{ \vec{w} \in \spn(\lat) \ : \ \forall \vec{y} \in \lat,\ \inner{\vec{w}, \vec{y}} \in \Z\}
\]
is the set of all vectors in the span of $\lat$ that have integer inner products with all lattice vectors. 
One can check that $\lat^{**} = \lat$ and that $\lat^*$ is generated by $\basis^* := \basis (\basis^T \basis)^{-1}$.

We write
\[
\lambda_1(\lat) := \min_{\vec{y} \in \lat \setminus \{\vec0\}} \|\vec{y}\|
\]
for the length of the shortest non-zero lattice vector.
The \emph{covering radius} is 
\[
\mu(\lat) := \max_{\vec{t} \in \spn(\lat)} \min_{\vec{y} \in \lat} \| \vec{t} - \vec{y}\|
\; .
\]
It is not hard to show that $\lambda_1(\lat) \leq 2\mu(\lat)$ (e.g., by taking $\vec{t} = \vec{v}/2$, where $\vec{v} \in \lat$ has $\|\vec{v}\| = \lambda_1(\lat)$).

The \emph{determinant} of the lattice is given by $\det(\lat) := \sqrt{\det(\basis^T \basis)}$, or simply $|\det(\basis)|$ in the full-rank case. One can show that the determinant is well defined (i.e., it does not depend on the choice of basis $\basis$). It follows that if $\lat \subset \R^n$ has full rank and $A \in \R^{n\times n}$ is non-singular, then $\det(A \lat) = |\det(A)| \det(\lat)$, and that $\det(\lat^*) = 1/\det(\lat)$.

A \emph{sublattice} $\lat' \subseteq \lat$ is an additive subgroup of $\lat$. We say that $\lat'$ is \emph{primitive} if $\lat' = \lat \cap \spn(\lat')$. 
For a primitive sublattice $\lat' \subseteq \lat$, we define the quotient lattice $\lat/\lat' := \pi_{\spn(\lat')^\perp}(\lat)$ to be the projection of $\lat$ onto the space orthogonal to $\lat'$. In particular, $\lat/\lat'$ is a lattice, and we have the identities $(\lat/\lat')^* = \lat^* \cap \spn(\lat')^\perp$ and $\det(\lat/\lat') = \det(\lat)/\det(\lat')$.

For a parameter $s > 0$ and $\vec{x} \in \R^n$, we define $\rho_s(\vec{x}) = e^{-\pi \|\vec{x}\|^2/s^2}$. Then, for any discrete set $A$, we define its Gaussian mass as $\rho_s(A) = \sum_{\vec{x} \in A} \rho_s(\vec{x})$. When $s = 1$, we omit the subscript.

We recall the Poisson Summation Formula for the Gaussian mass of a lattice, which says that
\begin{equation}
\label{eq:PSF}
\rho_s(\lat) = \frac{s^n}{\det(\lat)} \cdot \rho_{1/s}(\lat^*)
\end{equation}
for any $s > 0$ and (full-rank) lattice $\lat \subset \R^n$.
As an example, it follows that for any full-rank lattice $\lat \subset \R^n$, 
\begin{equation}
\label{eq:lowerboundonrho}
\rho_s(\lat) %
> \frac{s^n}{\det(\lat)} \; .
\end{equation}

\begin{lemma}[{\cite[Lemma 1.5]{banaszczyk}}]
	\label{lem:banaszczyk}
	For any lattice $\lat \subset \R^n$, shift vector $\vec{u} \in \R^n$, and any $r \geq 1/\sqrt{2\pi}$, 
	\[
	\rho((\lat - \vec{u}) \setminus r\sqrt{n}B_2^n) \leq \big(\sqrt{2\pi e r^2} e^{-\pi r^2}\big)^n \cdot \rho(\lat)
	\; .
	\]
\end{lemma}

The following claim was observed by Banaszczyk~\cite{banaszczyk} and is an immediate consequence of the Poisson Summation Formula and positivity.
\begin{claim}
	\label{clm:shifted_mass}
	For any lattice $\lat \subset \R^n$, shift vector $\vec{u} \in \R^n$, and parameter $s > 0$, $\rho_s(\lat - \vec{u}) \leq \rho_s(\lat)$ with equality if and only if $\vec{u} \in \lat$.
\end{claim}
\begin{proof}
    By applying the Poisson summation formula to the shifted Gaussian $\rho_s(\lat-\vec{u})$, we see that
    \[
        \rho_s(\lat- \vec{u}) = \frac{s^n}{\det(\lat)}\sum_{\vec{w} \in \lat^*} \rho_{1/s}(\vec{w}) \cos(2\pi \langle \vec{w}, \vec{u} \rangle) \leq \frac{s^n}{\det(\lat)}\sum_{\vec{w} \in \lat^*} \rho_{1/s}(\vec{w}) = \rho_s(\lat)
        \; ,
    \]
    as needed, where we have equality if and only if $\langle \vec{w}, \vec{u} \rangle \in \Z$ for all $\vec{w} \in \lat^*$, i.e., if and only if $\vec{u} \in \lat$.
\end{proof}

\begin{lemma}
	\label{lem:direct_sum_rho}
	For any lattice $\lat \subset \R^n$, primitive sublattice $\lat' \subset \lat$, and $s > 0$, 
	\[
	\rho_s(\lat) \leq \rho_s(\lat' \oplus \lat/\lat') = \rho_s(\lat')\rho_s(\lat/\lat')
	\; ,
	\]
	with equality if and only if $\lat = \lat' \oplus \lat/\lat'$.
\end{lemma}
\begin{proof}
	Let $T \subset \lat$ be any complete set of coset representatives of $\lat$ modulo $\lat'$. (In other words, for every $\vec{y} \in \lat$ there exists a unique $\vec{t} \in T$ such that $\vec{y} \equiv \vec{t} \pmod{\lat'}$.) Let $\pi := \pi_{\spn(\lat')}$, and $\pi^\perp := \pi_{\spn(\lat')^\perp}$. Then,
	\begin{align*}
	\sum_{\vec{y} \in \lat} \rho_s(\vec{y}) &= \sum_{\vec{t} \in T, \vec{y}' \in \lat'} \rho_s(\vec{t} + \vec{y}') \\
	&= 
    \sum_{\vec{t} \in T, \vec{y}' \in \lat'} \rho_s(\vec{y}' + \pi(\vec{t})) \rho_s(\pi^\perp(\vec{t}))\\
    &=  \sum_{\vec{t} \in T} \rho_s(\lat' + \pi( \vec{t})) \rho_s(\pi^\perp(\vec{t}))\\
    &\leq \rho_s(\lat') \sum_{\vec{t} \in T} \rho_s(\pi^{\perp}(\vec{t})) \\
    &= \rho_s(\lat') \rho_s(\lat/\lat')
		\; ,
	\end{align*}
	where the inequality is Claim~\ref{clm:shifted_mass} and the last equality uses the primitivity of $\lat'$.
\end{proof}

\subsection{Linear algebra}

We write $\mathrm{SL}_n(\R)$ for the group of all $n \times n$ determinant-one real matrices. 
A matrix $U \in \mathrm{SL}_n(\R)$ is \emph{orthogonal} if $U^T U = I_n$. Equivalently, a matrix is orthogonal if its associated linear transformation is an isometry. (I.e., $\|U \vec{x}\| = \|\vec{x}\|$ for all $\vec{x} \in \R^n$.) 
We write $\|A\| := \sup_{\vec{x} \in \R^n \setminus \{\vec0\}} \|A\vec{x}\|/\|\vec{x}\|$ for the \emph{operator norm} of $A$.

We recall the definition of the \emph{matrix exponential},
\[
e^A := I_n + A + A^2/2 + A^3/6 + \cdots
\; ,
\]
for any matrix $A \in \R^{n \times n}$, and the identity $\det(e^A) = e^{\Tr(A)}$.
Every positive-definite matrix $A$ has a matrix logarithm $M$ such that $e^M = A$.

\subsection{A note on the topology of the space of determinant-one lattices}

Throughout this paper, we consider various topological notions over the space of determinant-one (full-rank) lattices in $\R^n$ (e.g., local maxima of functions over this space, compact sets, etc.). 
Formally, the space of determinant-one lattices is $\mathrm{SO}_n(\R) \backslash \mathrm{SL}_n(\R)/\mathrm{SL}_n(\Z)$, i.e., the set of determinant-one real matrices modulo the orthogonal matrices $\mathrm{SO}_n(\R)$ (i.e., isometries) and the unimodular matrices $\mathrm{SL}_n(\Z)$, which are transformations between bases of the same lattice. The topology is the quotient topology. (See~\cite[Section 1.4]{TerrasBook}.)
However, the reader may prefer to think of the space of determinant-one lattice \emph{bases}, which is simply $\mathrm{SL}_n(\R)$ with its standard topology. 

\subsection{Stability}
\label{sec:stable}

We say that a lattice $\lat \subset \R^n$ is \emph{stable} if $\det(\lat) = 1$ and $\det(\lat') \geq 1$ for all sublattices $\lat' \subseteq \lat$. (Some authors call such lattices ``semistable.'') Note the obvious relationship between this notion and Theorem~\ref{thm:RM}. Here, we describe the properties of stable lattices that we will need in the sequel, and include proofs for completeness. 
This theory was developed by~\cite{HN75,Stuhler76,Grayson84}. See, e.g.,~\cite{Grayson84,Casselman04} for a more thorough treatment.

\begin{figure}
	\begin{center}
		\includegraphics{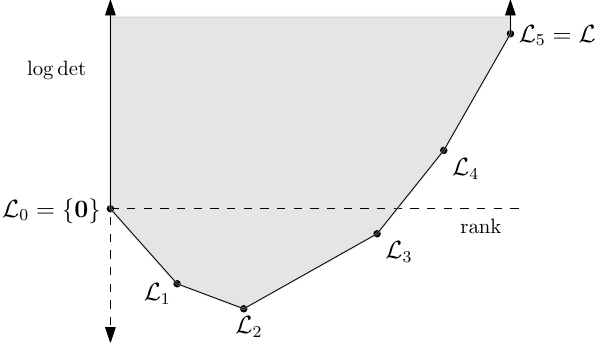}
		\caption{\label{fig:canonical_plot} 
			The canonical polygon of a (hypothetical) lattice $\lat$.}
	\end{center}
\end{figure}

We can in some sense ``decompose'' any lattice into stable lattices. To see this, we consider the two-dimensional scatter plot with points
\[
\{ (\rank(\lat'),\ \log \det(\lat')) \ : \ \lat' \subseteq \lat\}
\; ,
\] 
for some lattice $\lat \subset \R^n$,
where we explicitly include the trivial sublattice $\{\vec0\}$ and define $\log\det(\{\vec0 \}) := 0$. We call this the \emph{canonical plot} of $\lat$. Note that these points are bounded from below and  that for each rank $0 \leq k \leq n$, there is a (not necessarily unique) sublattice of rank $k$ that has minimal determinant among all sublattices with rank $k$. The convex hull of these points is therefore a degenerate polygon (bounded from below, but unbounded from above), called the \emph{canonical polygon} of $\lat$. See Figure~\ref{fig:canonical_plot}.

We are interested in the extremal points of this polygon. (E.g., $\lat_0,\ldots, \lat_5$ in Figure~\ref{fig:canonical_plot}.) Notice that any sublattice corresponding to an extremal point must necessarily be densest among all sublattices of that rank, but this necessary condition is not sufficient.  In fact, it is known that each extremal point corresponds to a \emph{unique} sublattice. Moreover, if two sublattices $\lat_1$ and $\lat_2$ both correspond to extremal points, then one is contained in the other. Therefore, the extremal points define a \emph{canonical filtration} of $\lat$, 
\[
\{ \vec0\} = \lat_0 \subset \lat_1 \subset \cdots \subset \lat_k = \lat
\; .
\]
(Note that we only include in the canonical filtration lattices that correspond to \emph{extremal points}, not any lattice on the boundary. E.g., the canonical filtration of $\Z^n$ is trivial: $\{\vec0\} = \lat_0 \subset \lat_1 = \Z^n$.)
It is also known that all of the quotients $\lat_i/\lat_{i-1}$ of adjacent sublattices in the canonical filtration are scalings of stable lattices. This is what we mean when we say that we can ``decompose'' a lattice into a sequence of stable lattices. 

Following~\cite{Grayson84,Casselman04}, 
we make the above (and other) facts precise in Proposition~\ref{prop:stable_properties}, which lists basic properties of the canonical filtration and stable lattices. We first need the following lemma, due to Stuhler~\cite{Stuhler76}. 
 
 \begin{lemma}
 	\label{lem:uncrossing}
 	For any $\lat \subset \R^n$ and any two primitive sublattices $\lat_1, \lat_2 \subseteq \lat$,
 	\[
 	\rank(\lat_1) + \rank(\lat_2) = \rank(\lat_1 \cap \lat_2) + \rank(\lat_1 + \lat_2)
 	\; ,
 	\]
 	and
 	\[
 	\det(\lat_1 \cap \lat_2)\det(\lat_1 + \lat_2) \leq \det(\lat_1)\det(\lat_2)
 	\; ,
 	\]
 	where we define $\det(\{\vec0\}) = 1$.
\end{lemma}
\begin{proof}
	The equality of ranks follows by considering the dimensions of the subspaces spanned by $\lat_1$, $\lat_2$, $\lat_1 \cap \lat_2$, and $\lat_1 + \lat_2$. For the inequality, suppose that $\mathcal{M}_1, \mathcal{M}_2 \subseteq \mathcal{M}$ are sublattices such that $\mathcal{M}_1 \cap \mathcal{M}_2 = \{\vec0\}$ and $\mathcal{M}_1 + \mathcal{M}_2 = \mathcal{M}$. Then, we have
	\[
	\det(\mathcal{M}) = \det(\mathcal{M}_1) \cdot \det(\pi_{\spn(\mathcal{M}_1)^\perp}(\mathcal{M}_2)) \leq \det(\mathcal{M}_1) \det(\mathcal{M}_2)
	\; ,
	\]
	where we have used the fact that $\pi_{\spn(\mathcal{M}_1)^\perp}$ is a contraction that preserves the rank of $\M_2$.
	Plugging in $\mathcal{M} := (\lat_1 + \lat_2)/(\lat_1 \cap \lat_2)$, $\mathcal{M}_1 := \lat_1/(\lat_1 \cap \lat_2)$ and $\mathcal{M}_2 := \lat_2/(\lat_1 \cap \lat_2)$ gives
	\begin{align*}
	\det(\lat_1 + \lat_2)/\det(\lat_1 \cap \lat_2) &= \det( (\lat_1 + \lat_2)/(\lat_1 \cap \lat_2)) \\
	&\leq \det(\lat_1/(\lat_1 \cap \lat_2)) \det(\lat_2/(\lat_1 \cap \lat_2)) \\
	&= \det(\lat_1)\det(\lat_2)/\det(\lat_1 \cap \lat_2)^2
	\; .
	\end{align*}
	The result follows by rearranging.
\end{proof}

\begin{figure}
	\begin{center}
		\includegraphics{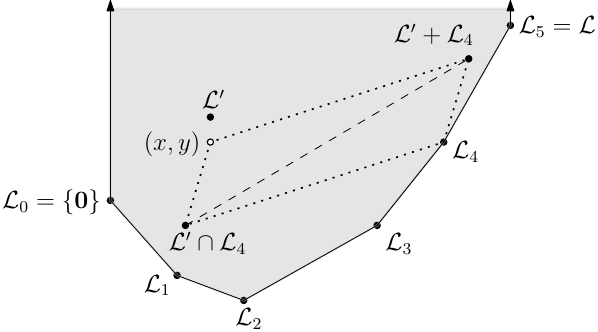}
		\caption{\label{fig:filtration_proof} 
		An illustration of the proof of Item~\ref{item:filtration} of Proposition~\ref{prop:stable_properties} (similar to~\cite[Figure 1.17]{Grayson84}). In particular, if $\lat' \not \subseteq \lat_4$, then $(x,y)$ must lie strictly above the dashed line, and therefore cannot be an extremal point (or even a boundary point) of the canonical polygon.
			}
	\end{center}
\end{figure}

\begin{proposition}
	\label{prop:stable_properties}
For any lattice $\lat \subset \R^n$, let $\{\vec0 \} = \lat_0, \lat_1, \ldots, \lat_k = \lat$ be all sublattices corresponding to extremal points of the canonical polytope, ordered by their rank. (See Figure~\ref{fig:canonical_plot}.) Then, 
\begin{enumerate}
	\item \label{item:filtration} the $\lat_i$ define a filtration $\lat_0 \subset \lat_1 \subset \cdots \subset \lat_k$ (in particular, the $\lat_i$ have distinct ranks); 
	\item \label{item:stable_quotient} the quotient lattice $\lat_i/\lat_{i-1}$ is a scaling of a stable lattice for $1 \leq i \leq k$ (i.e., $\alpha_i \cdot \lat_i/\lat_{i-1}$ is stable,  where $\alpha_i := \det(\lat_i/\lat_{i-1})^{-1/\rank(\lat_i/\lat_{i-1})}$); and
	\item \label{item:increasing_slopes} for all $1 \leq i \leq k-1$, $\det(\lat_i/\lat_{i-1})^{1/\rank(\lat_i/\lat_{i-1})} < \det(\lat_{i+1}/\lat_{i})^{1/\rank(\lat_{i+1}/\lat_{i})}$.
\end{enumerate}
Furthermore, 
\begin{enumerate}[label=(\roman*)]
	\item  \label{item:dual_stable} the dual of a stable lattice is stable;
	\item \label{item:compact} the set of all stable lattices is compact; 
	\item \label{item:direct_sum_stable} the direct sum of stable lattices is stable; and
	\item \label{item:stable_boundary} a lattice $\lat \subset \R^n$ is on the boundary of the set of stable lattices if and only if $\lat$ is stable and there is a primitive sublattice $\lat' \subset \lat$ with $0 < \rank(\lat') < n$ such that $\lat'$ and $\lat/\lat'$ are both stable.
\end{enumerate}
\end{proposition}
\begin{proof}
	To prove Item~\ref{item:filtration}, fix an index $i$ and
	let $\lat' \subset \lat$ be a sublattice with $\rank(\lat') \leq \rank(\lat_i)$ but $\lat' \not \subseteq \lat_i$. Below we will show that $\lat'$ must lie in the interior of the canonical polygon. In particular, this would imply that each $\lat_j$ with $j \leq i$ must have $\lat_j \subseteq \lat_i$, as desired.
	
	First, notice that we must have $\rank(\lat' + \lat_i) > \rank(\lat_i)$, since otherwise $\lat' + \lat_i$ would be a strict superlattice of $\lat_i$ with the same rank, contradicting the assumption that $\lat_i$ is an extremal point of the canonical polygon. Notice that, by Lemma~\ref{lem:uncrossing},  this also implies that 
	$\rank(\lat' \cap \lat_i) < \rank(\lat') \le \rank(\lat_i)$.
	Now, consider the point $(x,y)$ in the plane with 
	\[
	    x := \rank(\lat' \cap \lat_i) + \rank(\lat' + \lat_i) - \rank(\lat_i)
	\]
	and
	\[
	    y := \log \det (\lat' \cap \lat_i) + \log \det(\lat' + \lat_i) - \log \det(\lat_i)
	    \; .
	\]
	See Figure~\ref{fig:filtration_proof}. By Lemma~\ref{lem:uncrossing}, we have $\rank(\lat') = x$ and $\log \det(\lat') \geq y$, i.e., that the point $(\rank(\lat'), \log \det(\lat'))$ lies on or above the point $(x,y)$ in the plane. It therefore suffices to show that $(x,y)$ is in the interior of the canonical polygon, which we do by showing that $(x,y)$ lies strictly above the line segment between the points $(\rank(\lat' \cap \lat_i), \log \det(\lat' \cap \lat_i))$ and $(\rank(\lat' + \lat_i), \log \det(\lat' + \lat_i))$. (See the dashed line in Figure~\ref{fig:filtration_proof}.) Equivalently, it is enough to show that $(\rank(\lat_i), \log \det(\lat_i))$ lies strictly below this line segment (since it is the reflection of $(x,y)$ through the midpoint of the line segment). 
	This holds because $\lat_i$ corresponds to an extremal point, and, because $\rank(\lat' \cap \lat_i) < \rank(\lat_i) < \rank(\lat' + \lat_i)$, it is distinct from the two endpoints of the line segment. 
	
	To prove Item~\ref{item:stable_quotient}, let $\lat' \subseteq \lat_i/\lat_{i-1}$ be a sublattice. Let $\widehat{\lat} \subseteq \lat_i$ be the sublattice satisfying $\lat_{i-1} \subseteq \widehat{\lat}$ and $\lat' = \widehat{\lat}/\lat_{i-1}$. Since $\lat_{i-1}$ and $\lat_i$ are consecutive extremal points of the canonical polygon, the point $(\rank(\widehat{\lat}),\log \det(\widehat{\lat}))$ must lie on or above the line between $(\rank(\lat_{i-1}), \log \det(\lat_{i-1}))$ and $(\rank(\lat_{i}), \log \det(\lat_{i}))$. 
	This statement is equivalent to the inequality in the following:
	\begin{align*}
	\det(\lat') &= \det(\widehat{\lat})/\det(\lat_{i-1}) \\
				&\geq \Big( \frac{\det(\lat_i)}{\det(\lat_{i-1})} \Big)^{\frac{\rank(\widehat \lat) - \rank(\lat_{i-1})}{\rank(\lat_i) - \rank(\lat_{i-1})}} \\
				&=	\det(\lat_i/\lat_{i-1})^{\frac{\rank(\lat')}{\rank(\lat_i/\lat_{i-1})}}
		\; .
	\end{align*}
	I.e., if we set $\alpha_i := \det(\lat_i/\lat_{i-1})^{-1/\rank(\lat_i/\lat_{i-1})}$, then $\det(\alpha_i \lat') \geq 1$. It follows that $\alpha_i \lat_i/\lat_{i-1}$ is stable, as claimed.
	
	Item~\ref{item:increasing_slopes} simply says that the slopes of the lines between extremal points on the canonical polytope are strictly increasing. This follows immediately from the definition of the canonical polytope. (See Figure~\ref{fig:canonical_plot}.)
	
		To prove Item~\ref{item:dual_stable}, let $\mathcal{M}  \subset \R^n$ be a stable lattice and let $\mathcal{M}' \subseteq \mathcal{M}^*$ be a primitive sublattice of the dual. We have
		\[
		\det(\mathcal{M}') = \frac{1}{\det(\mathcal{M}^*/\mathcal{M}')} = \det((\mathcal{M}^*/\mathcal{M}')^*) = \det(\mathcal{M} \cap \spn(\mathcal{M}')^\perp) \geq 1
		\; .
		\]
		Therefore, $\mathcal{M}^*$ is stable.
	
	Item~\ref{item:compact} follows immediately from Mahler's compactness theorem~\cite{Mahler46} together with the observation that $\lambda_1(\mathcal{M}) \geq 1$ for any stable lattice $\mathcal{M}$.

	To prove Item~\ref{item:direct_sum_stable}, let $\mathcal{M}_1, \mathcal{M}_2$ be two stable lattices, and let $\mathcal{M}' \subset \mathcal{M}_1 \oplus \mathcal{M}_2$ be a primitive sublattice. Then, applying Lemma~\ref{lem:uncrossing}, we have
	\[
	\det(\mathcal{M}') \geq \frac{\det(\mathcal{M}' \cap \mathcal{M}_1) \det(\mathcal{M}' + \mathcal{M}_1)}{\det(\mathcal{M}_1)} = \det(\mathcal{M}' \cap \mathcal{M}_1) \det(\mathcal{M}' + \mathcal{M}_1)
	\; .
	\]
	Note that $\mathcal{M}' \cap \mathcal{M}_1$ is a sublattice of $\mathcal{M}_1$, so that $\det(\mathcal{M}' \cap \mathcal{M}_1) \geq 1$. 
	And $\mathcal{M}' + \mathcal{M}_1 = \mathcal{M}_1 \oplus \pi_{\spn(\mathcal{M}_2)}(\mathcal{M}')$ is the direct sum of $\mathcal{M}_1$ with a sublattice of $\mathcal{M}_2$, so that $\det(\mathcal{M}' + \mathcal{M}_1) = \det(\pi_{\spn(\mathcal{M}_2)}(\mathcal{M}')) \geq 1$ as well. The result follows.

	Finally, Item~\ref{item:stable_boundary} follows by first noting that a stable lattice $\mathcal{M}$ is on the boundary if and only if there is some strict primitive non-zero sublattice $\mathcal{M}' \subset \mathcal{M}$ with $\det(\mathcal{M}') = 1$. Clearly, $\mathcal{M}'$ is stable, since it has determinant one and all of its sublattices are also sublattices of $\mathcal{M}$, so that they must have determinant at least one. 
	To see that $\mathcal{M}/\mathcal{M}'$ is stable, first notice that it has determinant one. Next, let $\lat' \subseteq \mathcal{M}/\mathcal{M}'$ be an arbitrary sublattice. Let $\widehat{\lat} \subseteq \mathcal{M}$ be the sublattice satisfying $\mathcal{M}' \subseteq \widehat{\lat}$ and $\lat' = \widehat{\lat}/\mathcal{M}'$. Then $\det \lat' = \det \widehat{\lat}/ \det \mathcal{M}' = \det \widehat{\lat} \ge 1$, where the inequality uses that $\mathcal{M}$ is stable. 
\end{proof}

\subsection{The Voronoi cell and fundamental bodies}

The \emph{Voronoi cell} of a lattice $\lat \subset \R^n$,
\[
\V(\lat) := \{\vec{x} \in \R^n \ : \ \inner{\vec{x}, \vec{y}} \leq \|\vec{y}\|^2/2,\ \forall \vec{y} \in \lat \}
\; ,
\]
is the set of vectors in $\R^n$ that are at least as close to $\vec0$ than to any other lattice vector. 
In fact, it is a symmetric polytope.

A \emph{fundamental body} of a lattice $\lat \subset \R^n$ is any convex body $K \subset \R^n$ such that $K + \lat = \R^n$  and
$\mathrm{Int}(K) \cap (K + \vec{y}) = \emptyset$ for any non-zero lattice point $\vec{y} \in \lat \setminus \{\vec0\}$. 
Equivalently, $\vol(K) = \det(\lat)$ and $\mathrm{Int}(K) \cap (K + \vec{y}) = \emptyset$ for any non-zero lattice point $\vec{y} \in \lat \setminus \{\vec0\}$.
In particular, the Voronoi cell is a fundamental body.

\begin{claim}
	\label{clm:fundamental_product}
	For any lattice $\lat \subset \R^n$, primitive sublattice $\lat' \subset \R^n$, fundamental body $K_1 \subset \spn(\lat')$ of $\lat'$, and fundamental body $K_2 \subset \spn(\lat')^\perp$ of $\lat/\lat'$,
	the Minkowski sum $K:= K_1 + K_2$ is a fundamental body of $\lat$. In particular, if $\{\vec0\} = \lat_0 \subset \lat_1 \subset \cdots \subset \lat_k $ is a filtration of primitive sublattices, then
	\[
	\V\Big(\bigoplus_i \lat_i/\lat_{i-1} \Big) = \V(\lat_1/\lat_0) + \cdots + \V(\lat_k/\lat_{k-1})
	\]
	is a fundamental body of $\lat$.
\end{claim}
\begin{proof}
	Notice that
	\[
	\vol(K) = \vol(K_1) \cdot \vol(K_2) = \det(\lat') \cdot \det(\lat/\lat') = \det(\lat) 
	\; .
	\]
	It therefore suffices to show that $\mathrm{Int}(K) \cap (K + \vec{y}) = \emptyset$ for any $\vec{y} \in \lat \setminus \{\vec0\}$.  So suppose there exists $\vec{y} \in \lat$ such that $\mathrm{Int}(K) \cap (K + \vec{y}) \neq \emptyset$. 
	Then, by projecting orthogonally to $\lat'$, we see that $\mathrm{Int}(K_2) \cap (K_2 + \pi_{\spn(\lat')^\perp}(\vec{y})) \neq \emptyset$. Since $K_2$ is a fundamental body of $\lat/\lat'$ and $\pi_{\spn(\lat')^\perp}(\vec{y}) \in \lat/\lat'$, it follows that $\pi_{\spn(\lat')^\perp}(\vec{y}) = \vec0$, i.e., $\vec{y} \in \lat'$. 
	Intersecting with $\spn(\lat')$, this implies that $\mathrm{Int}(K_1) \cap (K_1 + \vec{y}) \neq \emptyset$. Since $\vec{y} \in \lat'$ and $K_1$ is a fundamental body of $\lat'$, we obtain that $\vec{y} = \vec0$. The result follows.
\end{proof}

We will also need the following claim, which follows immediately from the definition of a fundamental body.

\begin{claim}
	\label{clm:transform_fund_body}
	For any lattice $\lat \subset \R^n$, fundamental body $K$ of $\lat$, and non-singular matrix $A \in \R^{n \times n}$, the body $A K $ is a fundamental body of $A \lat$. In particular, $A \V(\lat)$ is a fundamental body of $A \lat$.
\end{claim}
\begin{proof}
	It suffices to notice that $AK + A\lat = A(K + \lat) = \R^n$ and 
	$\mathrm{Int}(AK) \cap (A K + A\vec{y}) = A(\mathrm{Int}(K) \cap (K + \vec{y})) = \emptyset$
	for $\vec{y} \in \lat \setminus \{\vec0\}$.
\end{proof}

The next lemma and its corollary show that the Voronoi cell is in some sense the ``optimal fundamental body.'' They are very similar to some results due to Dadush~\cite[Lemma 6.3.6, Corollary 6.3.7]{thesis/D12}.

\begin{lemma}
	\label{lem:Voronoi_wins}
	For any lattice $\lat \subset \R^n$, there is a map $\psi_\lat : \R^n \to \V(\lat)$ such that $\|\psi_\lat(\vec{x})\| \leq \|\vec{x}\|$, and for every fundamental body $K$ of $\lat$, $\psi_\lat$ restricted to $\mathrm{Int}(K)$ is injective and volume-preserving.
\end{lemma}
\begin{proof}
	The function $\psi_\lat$ just maps $\vec{x}$ to a representative of $\vec{x} \bmod \lat$ that is in the Voronoi cell. Specifically, let $\mathsf{CVP}_\lat(\vec{x}) := \argmin_{\vec{y} \in \lat} \|\vec{y} - \vec{x}\|$ be a closest lattice vector to $\vec{x}$ (breaking ties arbitrarily), and let $\psi_\lat(\vec{x}) := \vec{x} - \mathsf{CVP}_\lat(\vec{x})$. By the definition of $\mathsf{CVP}$, it is immediate that $\|\psi_{\lat}(\vec{x})\| = \min_{\vec{y} \in \lat} \|\vec{y}- \vec{x}\| \leq \|\vec{x}\|$.
	
	Suppose $\psi_\lat(\vec{x}) = \psi_\lat(\vec{x}')$ for some $\vec{x}, \vec{x}' \in \mathrm{Int}(K)$. I.e., $\vec{x} - \mathsf{CVP}_\lat(\vec{x}) = \vec{x}' - \mathsf{CVP}_\lat(\vec{x}')$. Rearranging, we see that $\vec{y} := \vec{x} - \vec{x}' = \mathsf{CVP}_\lat(\vec{x}) - \mathsf{CVP}_\lat(\vec{x}')$ is a lattice point. But, $\vec{x} \in \mathrm{Int}(K) \cap (K + \vec{y})$.  Since $K$ is a fundamental body, it follows that $\vec{y} = \vec0$. I.e., $\vec{x} = \vec{x}'$, and $\psi_\lat$ is injective over $\mathrm{Int}(K)$.
	
	The fact that $\psi_\lat$ is volume-preserving over $\mathrm{Int}(K)$ follows from the fact that it is injective and preserves volume locally.
\end{proof}

\begin{corollary}
	\label{cor:monotone_Voronoi_wins}
	For any non-decreasing measurable function $f : \R_{\ge 0} \to \R$, lattice $\lat \subset \R^n$, and fundamental body $K$ of $\lat$, 
	\[
	\int_{\V(\lat)} f(\|\vec{x}\|) {\rm d} \vec{x}\leq \int_{K} f(\|\vec{x}\|) {\rm d} \vec{x}
	\; .
	\]
\end{corollary}
\begin{proof}
	\[
	\int_{K} f(\|\vec{x}\|) {\rm d} \vec{x} = 
	\int_{\mathrm{Int}(K)} f(\|\vec{x}\|) {\rm d} \vec{x} \geq 
	\int_{\mathrm{Int}(K)} f(\|\psi_\lat(\vec{x})\|) {\rm d} \vec{x} = \int_{\psi_\lat(\mathrm{Int}(K))} f(\|\vec{x}\|) {\rm d} \vec{x} = \int_{\V(\lat)} f(\|\vec{x}\|) {\rm d} \vec{x}
	\; ,
	\]
	where the last equality follows from the fact that $\psi_\lat$ preserves volume and $\vol(\mathrm{Int}(K)) = \vol(\V(\lat))$, so it must be the case that $\psi_\lat(\mathrm{Int}(K)) \subset \V(\lat)$ differs from $\V(\lat)$ on a set of measure zero.
\end{proof}

\subsection{Matrix calculus}
\label{sec:matrix_calc}

We say that a function $g : \R^{n \times n} \to \R$ is \emph{differentiable} at $Q \in \R^{n\times n}$ 
if there exists a $B \in \R^{n \times n}$ such that
\begin{equation*}
\lim_{M \to 0} \frac{g(Q+  M) - g(Q) - \Tr(B^TM)}{\|M\|} = 0
\; ,
\end{equation*}
and we call $B$ the \emph{gradient} of $g$ at $Q$, 
\begin{equation}
\label{eq:gradient}
\grad_A g(A)|_{A = Q} := B
\; .
\end{equation}
(Some authors prefer to define $\grad_A g(A)|_{A = Q}$ as $B^T$.)

\section{Gradients over lattices and over positions of the Voronoi cell}
\label{sec:gradient}

The purpose of this section is to prove the following theorem.

\begin{theorem}
	\label{thm:same_grad_VL}
	For any continuously differentiable function $f : \R_{\geq 0} \to \R$ and lattice $\lat \subset \R^n$, let
\[
	g(A) := \frac{1}{|\det(A)|} \cdot \int_{\V(A \lat)} f(\|\vec{x}\|^2){\rm d}\vec{x} \text{, \ and \  } h(A) := \frac{1}{|\det(A)|} \cdot \int_{ A\V(\lat)} f(\|\vec{x}\|^2){\rm d}\vec{x}
	\; ,
\]
	where $A \in \R^{n \times n}$ is a non-singular matrix.
	Then, $g$ and $h$ are differentiable at $A = I_n$, with
	\[
	\grad_A g(A)|_{A = I_n} = \grad_A h(A)|_{A = I_n} = 2\int_{\V(\lat)} f'(\|\vec{x}\|^2) \vec{x}\vec{x}^T{\rm d}\vec{x}
	\; ,
	\]
	where $f'(x) := \frac{{\rm d}}{{\rm d}x}f(x)$.
\end{theorem}

We first compute the gradient of $h$, which is straightforward.

\begin{claim}
	\label{clm:grad_h}
	For any continuously differentiable function $f : \R_{\geq 0} \to \R$ and bounded measurable set $U \subset \R^n$, let
	\[
	h(A) := \frac{1}{|\det(A)|} \cdot \int_{A U} f(\|\vec{x}\|^2){\rm d}\vec{x}
	\; ,
	\]
	where $A \in \R^{n \times n}$ is a non-singular matrix. Then, $h$ is differentiable with
	\[
	\grad_A h(A) |_{A = I_n}= 2\int_U  f'(\|\vec{x}\|^2)\vec{x} \vec{x}^T {\rm d}\vec{x}
	\; ,
	\]
	where $f'(x) := \frac{{\rm d}}{{\rm d}x}f(x)$.
\end{claim}
\begin{proof}
	By a change of variables, we have
	\[
	h(A) = \int_{U} f(\|A\vec{x}\|^2){\rm d}\vec{x} 
	\; .
	\]
	Next, by the chain rule, $\grad_A f(\|A\vec{x}\|^2) = 2f'(\|A\vec{x}\|^2) A\vec{x} \vec{x}^T$, which is bounded as $x$ ranges over $U$ and $A$ ranges over any bounded set. Therefore, we may use the bounded convergence theorem to swap the gradient and the integral and write 
	\begin{align*}
	\grad_A h(A) 
	&= \int_U \big( \grad_A f(\|A\vec{x}\|^2) \big) {\rm d}\vec{x}\\
	&= 2\int_U  f'(\|A\vec{x}\|^2) A\vec{x} \vec{x}^T  {\rm d}\vec{x} %
	\; . &&\qedhere
	\end{align*}
\end{proof}

Now, we prove Theorem~\ref{thm:same_grad_VL}. 
(We thank Ronen Eldan for showing us this proof. An earlier version of this work had a much longer proof.)

\begin{proof}[Proof of Theorem~\ref{thm:same_grad_VL}]
Fix a continuously differentiable function $f : \R_{\ge 0} \to \R$ and a lattice $\lat \subset \R^n$. By Claim~\ref{clm:grad_h}, it suffices to show that the gradient of $g$ at $I_n$ exists and is equal to the gradient of $h$ at $I_n$. 
Recalling the definition of the gradient in Eq.~\eqref{eq:gradient}, we see that this is equivalent to proving that
\begin{equation}
	\label{eq:same_grad_no_det}
\lim_{M \to 0} \frac{h(I_n + M) - g(I_n + M)}{\|M\|} = 0
\; .
\end{equation}
In fact, setting $\nu = \nu(\lat) > 1$ as in Claim~\ref{clm:eps_voronoi} below, we will show that for any $M \in \R^{n \times n}$ with operator norm $\|M\| < 1/(n \nu)$,
\begin{equation}
\label{eq:same_grad}
|\det(I_n + M)| \cdot \big|h(I_n + M) -g(I_n + M)\big| \leq C^* \|M\|^2
\; ,
\end{equation}
where $C^* := C^*(\lat, f) > 0$ is independent of $M$. This implies Eq.~\eqref{eq:same_grad_no_det} because the determinant is bounded away from zero in a neighborhood around $I_n$, so that $\lim_{M \to 0} \|M\|^2/( \|M\| \cdot  |\det(I_n + M)|) = 0$.

Let $A := I_n + M$. By the definition of the Voronoi cell, $\|\vec{x}\| = \dist(\vec{x}, A \lat)$ if $\vec{x} \in \V(A \lat)$. Therefore, 
\begin{align}
|\det(A)| \cdot g(A) 
&=  \int_{\V(A \lat)} f(\|\vec{x}\|^2){\rm d}\vec{x} \nonumber \\
&=\int_{\V(A \lat)} f(\dist(\vec{x}, A\lat)^2){\rm d}\vec{x} \nonumber  \\
&=  \int_{A \V( \lat)} f(\dist(\vec{x}, A \lat)^2){\rm d}\vec{x}
\; , \label{eq:detgasintegral}
\end{align}
where the last equality uses the facts that (1) $A \V(\lat)$ is a fundamental domain of $A \lat$ (i.e., Claim~\ref{clm:transform_fund_body}), and (2) the distance function $\vec{x} \mapsto \dist(\vec{x}, A \lat)$ is periodic over $A \lat$ so that the integral is the same over any fundamental domain.

Now, let $\widehat{\V} := (A \V(\lat)) \setminus \V(A \lat)$ be the set of points in $A \V(\lat)$ such that $\|\vec{x}\| \neq \dist(\vec{x}, A \lat)$. Using Eq.~\eqref{eq:detgasintegral}, we see that the left-hand side of Eq.~\eqref{eq:same_grad} is
\begin{align*}
\Big| \int_{A\V( \lat)} \big( f(\|\vec{x}\|^2) - f(\dist(\vec{x}, A\lat)^2) \big) {\rm d}\vec{x} \Big| 
&= \Big| \int_{\widehat{\V}} \big( f(\|\vec{x}\|^2) - f(\dist(\vec{x}, A \lat)^2) \big) {\rm d}\vec{x} \Big| \\
&\leq \vol(\widehat{\V}) \cdot \max_{\vec{x} \in \widehat{\V}} \big|f(\|\vec{x}\|^2) - f(\dist(\vec{x}, A\lat)^2) \big|
\; .
\end{align*}
We complete the proof by arguing that $\vol(\widehat{\V} ) \leq C_0^* \|M\|$ and $\max_{\vec{x} \in \widehat{\V}} \big|f(\|\vec{x}\|^2) - f(\dist(\vec{x}, A \lat)^2) \big| \leq C_1^* \|M\|$, where $C_0^* := C_0^*(\lat) > 0$ and $C_1^* := C_1^*(\lat, f) >  0$ are independent of $M$. 

By Claim~\ref{clm:eps_voronoi},  
\[
\widehat{\V} = (A\V(\lat)) \setminus \V(A \lat) \subset \big((1+\nu \|M\|)\cdot \V(\lat)\big)\setminus \big((1- \nu \|M\|) \cdot \V(\lat)\big)
\; ,
\]
and therefore,
\begin{align*}
\vol(\widehat{\V} ) 
&\le 
\big((1+ \nu \|M\|)^n - (1- \nu \|M\|)^n \big) \cdot  \vol(\V(\lat)) \\
& \le
100n \nu \vol(\V(\lat)) \cdot \|M\| \; ,
\end{align*}
as desired, where we have used that $\|M\| < 1/(n \nu)$.

Next, notice that for $\vec{x} \in \widehat{\V} \subset A\V(\lat)$, we have, say, $\|\vec{x}\| \leq \|A\| \mu(\lat) \leq 2 \mu(\lat)$ and $\dist(\vec{x}, A\lat) \leq \|\vec{x}\| \leq 2\mu(\lat)$. Therefore, 
\begin{align}
	\big|f(\|\vec{x}\|^2) - f(\dist(\vec{x}, A \lat)^2) \big| 
		&\leq \big|\|\vec{x}\| - \dist(\vec{x}, A \lat)\big| \cdot \max_{0 \leq r \leq 2\mu(\lat)} \Big|\frac{\rm d}{{\rm d} r} f(r^2) \Big| \nonumber \\
		&= 2 \big|\|\vec{x}\| - \dist(\vec{x}, A \lat)\big| \cdot \max_{0 \leq r \leq 2\mu(\lat)} r|f'(r^2)|
	\; ,\label{eq:boundc1}
\end{align}
where the inequality follows, e.g., from the fundamental theorem of calculus. Since $f'$ is continuous by assumption, the maximum in Eq.~\eqref{eq:boundc1} is well defined and finite. 
To complete the proof, let 
$\vec{x}' := (1-\nu \|M\|)^2\vec{x}$, 
which satisfies 
$\vec{x}'  \in \V(A \lat)$ because
\[
(1-\nu \|M\|)^2 A \V(\lat) \subset
(1+\nu \|M\|) (1-\nu\|M\|) \V(A \lat) \subset
\V(A \lat)
\]
where we used the right containment in Eq.~\eqref{eq:Voronoi_claimed_inclusion_1}
and the left containment in Eq.~\eqref{eq:Voronoi_claimed_inclusion_2}
of Claim~\ref{clm:eps_voronoi} below.
 Then,
by the triangle inequality,
\begin{align*}
\big| \|\vec{x}\| - \dist(\vec{x}, A \lat)\big| 
	&\leq 2\big\|\vec{x} - \vec{x}'\big\| + \big| \|\vec{x}'\| - \dist(\vec{x}', A \lat)\big|\\
	&= 2\big\|\vec{x} - \vec{x}'\big\|\\
	&\leq 4 \nu \|\vec{x}\| \|M\| \\
	&\leq 8 \nu \mu(\lat) \|M\|\; . 
\end{align*}
Combining this with Eq.~\eqref{eq:boundc1} shows that there exists a $C_1^*$ of the desired form.
\end{proof}

\begin{claim}
	\label{clm:eps_voronoi}
	For any lattice $\lat \subset \R^n$ there exists $\nu = \nu(\lat) > 1$ such that for all $M \in \R^{n \times n}$ with $\|M\| < 1/\nu$,
	we have
	\begin{equation}
	\label{eq:Voronoi_claimed_inclusion_1}
	(1- \nu\|M\|)\cdot \V(\lat)  \subset A \V(\lat) \subset (1+ \nu \|M\|) \cdot \V(\lat) 
	\; ,
	\end{equation}
	and similarly,
	\begin{equation}
	\label{eq:Voronoi_claimed_inclusion_2}
	(1- \nu\|M\|) \cdot \V(\lat) \subset  \V(A \lat) \subset (1+ \nu \|M\|) \cdot \V(\lat) 
	\; ,
	\end{equation} 
	where $A := I_n+M$.
\end{claim}
\begin{proof}
	We take $\nu := 20\mu(\lat)/\lambda_1(\lat)$
	and notice that $\nu \geq 10$. Then,
	for any $\vec{x}\in \V(\lat) \subset \mu(\lat) B_2^n$, we have $\| A \vec{x} - \vec{x}\| = \|M \vec{x}\| \leq \mu(\lat) \|M\| $. Since $(\lambda_1(\lat)/2) \cdot B_2^n \subset \V(\lat)$, we have
	\begin{equation*}
	A \V(\lat) \subset \V(\lat) +  \mu(\lat)\|M\| B_2^n  \subset (1+ \nu\|M\|) \cdot  \V(\lat)
	\; .
	\end{equation*}
	Similarly, 
	\[
	\V(\lat) \subset A\V(\lat) + \mu(\lat) \|M\| B_2^n \subset (1-\nu \|M\|)^{-1}\cdot A\V(\lat)
	\; ,
	\]
	where the second inclusion uses the fact that, say, $(\lambda_1(\lat)/4) \cdot B_2^n \subset A\V(\lat)$ since $\|A^{-1}\| \leq 2$. This establishes Eq.~\eqref{eq:Voronoi_claimed_inclusion_1}.
	
	Next, let $\vec{x} \in \V(\lat)$. Let $\alpha := 2\max_{\vec{y}\in \lat \setminus \{\vec0\}} \inner{A \vec{y}, \vec{x}}/\|A\vec{y}\|^2$. Equivalently, $\alpha > 0$ is minimal such that $\vec{x}/\alpha \in  \V(A \lat)$. 
We claim that $\alpha \leq 1+ \nu \|M\|$, which implies that $ (1 - \nu \|M\|) \cdot  \V(\lat) \subset \V(A\lat) $.
	Indeed, for non-zero $\vec{y} \in \lat$, we have
	\begin{align*}
	\frac{\inner{A \vec{y}, \vec{x}}}{\|A\vec{y}\|^2} 
	& \leq \frac{\inner{\vec{y}, \vec{x}} + \inner{M \vec{y}, \vec{x}}}{(1-\|M\|)^2 \|\vec{y}\|^2} \\
	& \leq 
	\frac{|\inner{\vec{y}, \vec{x}}|}{(1-\|M\|)^2 \|\vec{y}\|^2} + \frac{\|M\|}{(1-\|M\|)^2} \cdot \frac{\|\vec{x}\|}{\|\vec{y}\|}  \\ 
	& \leq \frac{|\inner{\vec{y}, \vec{x}}|}{(1-\|M\|)^2 \|\vec{y}\|^2} + \frac{\|M\|}{(1-\|M\|)^2} \cdot \frac{\|\vec{x}\|}{\lambda_1(\lat)}
	\; .
	\end{align*}
	Since $\vec{x} \in \V(\lat)$, $|\inner{\vec{y}, \vec{x}}|/\|\vec{y}\|^2 \leq 1/2$ and $\|\vec{x}\| \leq \mu(\lat)$. Therefore, 
	\[
	\alpha \leq \frac{1}{(1-\|M\|)^2} + \frac{2\|M\| }{(1-\|M\|)^2} \cdot \frac{\mu(\lat)}{\lambda_1(\lat)} \leq 1 + \nu \|M\| 
	\; ,
	\]
	as claimed.
	
	Finally, we use a similar argument to prove that $\V(A \lat)  \subset(1+ \nu \|M\|) \cdot \V(\lat)$. In particular, for any $\vec{x} \in \V(A\lat)$ and non-zero $\vec{y} \in \lat$, we have 
	\begin{align*}
	    \frac{\inner{\vec{y}, \vec{x}}}{\|\vec{y}\|^2} 
	        &\leq \frac{\inner{A \vec{y}, \vec{x}}}{\|\vec{y}\|^2} + \|M\|\cdot\frac{\|\vec{x}\|}{\|\vec{y}\|} \\
	        &\leq   (1+\|M\|)^2 \cdot  \frac{\inner{A\vec{y}, \vec{x}}}{\|A\vec{y}\|^2} + \|M\| \cdot \frac{\mu(A\lat)}{\lambda_1(\lat)}\\
	        &\leq \frac{(1+\|M\|)^2}{2} + \|M\| \cdot \frac{\mu(A\lat)}{\lambda_1(\lat)}\\
	        &\leq \frac{1 + \nu \|M\|}{2}
	    \; ,
	\end{align*}
	which implies the result.
	For this last inequality, we use the fact that, e.g., $\mu(A\lat) \leq 2\mu(\lat)$ since 
	\[
		\mu(A\lat) = \max_{\vec{x} \in \R^n} \min_{\vec{y}\in\lat} \|A\vec{x} - A \vec{y} \| \leq \|A\| \cdot \max_{\vec{x} \in \R^n} \min_{\vec{y}\in\lat} \|\vec{x} - \vec{y} \| \leq 2 \mu(\lat)
		\; .
		\qedhere
	\]
\end{proof}

\section{Proof of the Reverse Minkowski Theorem}
\label{sec:voronoi_mass}

In this section, we prove our main theorem, Theorem~\ref{thm:RM}. 
Recall that the Voronoi cell $\V(\lat)$ of a lattice $\lat \subset \R^n$ is the symmetric polytope of all vectors in $\R^n$ that are closer to $\vec0$ than to any other lattice vector,
\[
\V(\lat) := \{ \vec{x} \in \R^n \ : \ \forall \vec{y} \in \lat ,\ \inner{\vec{y}, \vec{x}} \leq \|\vec{y}\|^2/2 \}
\; .
\]
Also recall that for parameter $s > 0$, $\gamma_s(\cdot)$ is the Gaussian measure on $\R^n$ given by
\[
\gamma_s(S) := \int_{S/s} e^{-\pi \|\vec{x}\|^2} {\rm d} \vec{x}
\] 
for any measurable set $S \subseteq \R^n$. (Some authors prefer to parametrize $\gamma$ in terms of the standard deviation $\sigma := s/\sqrt{2\pi}$.)
We omit the subscript when $s=1$. 
We are interested in the Gaussian mass $\gamma_s(\V(\lat))$ of the Voronoi cell because, as the following lemma due to Chung, Dadush, Liu, and Peikert shows, this can be used to obtain an upper bound on the mass $\rho_s(\lat)$ of the lattice itself~\cite{CDLP12}. We include a proof for completeness. 

\begin{lemma}[{\cite[Lemma 3.4]{CDLP12}}]
	\label{lem:rho_gamma}
	For any lattice $\lat \subset \R^n$ and $s > 0$, 
	\[
	\rho_s(\lat) \cdot \gamma_s(\V(\lat)) \leq 1
	\; .
	\]
\end{lemma}
\begin{proof}
	By scaling appropriately, we may assume without loss of generality that $s = 1$. Note that the Voronoi cell tiles space with respect to $\lat$. I.e., $\bigcup_{\vec{y} \in \lat} (\V(\lat) + \vec{y}) = \R^n$, where the union is disjoint except on a measure-zero set.  So,
	\begin{align*}
	1 &= \int_{\R^n} e^{-\pi \|\vec{x}\|^2} {\rm d}\vec{x} \\
	&= \sum_{\vec{y} \in \lat} \int_{\V(\lat)} e^{-\pi \|\vec{y} + \vec{t}\|^2} {\rm d}\vec{t} \\
	&= \sum_{\vec{y} \in \lat} e^{-\pi \|\vec{y}\|^2} \int_{\V(\lat)}  e^{-\pi \|\vec{t}\|^2} e^{2\pi \inner{\vec{y}, \vec{t}}} {\rm d}\vec{t} \\
	&= \sum_{\vec{y} \in \lat} \rho(\vec{y}) \int_{\V(\lat)} e^{-\pi \|\vec{t}\|^2} \cosh(2\pi \inner{\vec{y}, \vec{t}}){\rm d}\vec{t} \\
	&\geq \sum_{\vec{y} \in \lat} \rho(\vec{y}) \int_{\V(\lat)} e^{-\pi \|\vec{t}\|^2}{\rm d}\vec{t} \\
	&= \rho(\lat)\gamma(\V(\lat)) 
	\;,
	\end{align*}
	where the fourth line follows from the fact that the Voronoi cell is symmetric.
\end{proof}

Therefore, in order to prove Theorem~\ref{thm:RM}, it suffices to show that $\gamma_{1/t}(\V(\lat)) \geq 2/3$ for every lattice $\lat \subset \R^n$ with $\det(\lat') \geq 1$ for all sublattices $\lat' \subseteq \lat$, where $t:=10(\log n + 2)$. 
As we explained in the introduction, we will reduce this to studying local minima of the function $\lat \mapsto \gamma_{1/t}(\V(\lat))$ over the set of determinant-one lattices. (We do not know whether such local minima actually exist.)

In Section~\ref{sec:gaussianconvex}, we collect some (mostly known) facts about the Gaussian mass of convex bodies. The statements of Theorem~\ref{thm:isotropic_mass} and Lemma~\ref{lem:gamma_concentration} are the only 
parts of Section~\ref{sec:gaussianconvex} that will be used later on. In particular, in Section~\ref{sec:actualproofofmain}, we apply these two facts to prove Theorem~\ref{thm:RM}.

\subsection{Gaussian mass of convex bodies}
\label{sec:gaussianconvex}

We say that a measurable set $U \subset \R^n$ is in \emph{isotropic Gaussian position for parameter $s$} if
\[
\int_{U/s} e^{-\pi \|\vec{x}\|^2} \vec{x}\vec{x}^T {\rm d}\vec{x} = \alpha \cdot I_n 
\]
for some scalar $\alpha > 0$. 
If $s = 1$, we simply say that $U$ is in \emph{isotropic Gaussian position}. Such a position has been considered elsewhere (e.g.,~\cite{bobkov11}), but as far as we know, it did not previously have a name.

The main goal of this section is to prove the following theorem. We will also
include a standard fact in Lemma~\ref{lem:gamma_concentration} towards the end
of this section.

\begin{theorem}
	\label{thm:isotropic_mass}
	For any symmetric convex body $K \subset \R^n$ with $\vol(K) \geq 1$, if $K$ is in isotropic Gaussian position for some 
	parameter $0 < s \leq 1/t$, then $\gamma_{s}(K) \geq 2/3$ where 
	$t := 10(\log n + 2)$.
\end{theorem}

Our proof of Theorem~\ref{thm:isotropic_mass} proceeds in two parts. 
The first part is a result due to Bobkov~\cite{bobkov11} (Proposition~\ref{prop:maximum_mass} below), showing 
that an isotropic Gaussian position of a convex body has maximal Gaussian mass.
We include a proof for completeness. 
In the second part  (Theorem~\ref{thm:ell_position_mass} below), we show that any volume-one convex body $K \subset \R^n$ has a position such that $\gamma_{s}(K) \geq 2/3$.

\begin{proposition}[{\cite[Proposition 3.1]{bobkov11}}]
\label{prop:maximum_mass}
For any symmetric convex body $K \subset \R^n$, if $K$ is in isotropic Gaussian position for some parameter $s > 0$, then 
$\gamma_s(K) \geq \gamma_s(A K)$ for any determinant-one matrix $A \in \mathrm{SL}_n(\R)$.
\end{proposition}

We start by observing that isotropic Gaussian positions correspond to critical points of the Gaussian mass function over positions.
\begin{fact}
	\label{fact:gaussian_gradient}
	For any measurable set $U \subset \R^n$, let
	\[
	h(A) := \frac{\gamma(AU)}{|\det(A)|}
	\; ,
	\]
where $A \in \R^{n \times n}$ is a non-singular matrix.
	Then,
	\[
	\grad_A h(A)|_{A = I_n} = -2\pi \int_{U} e^{-\pi \|\vec{x}\|^2} \vec{x} \vec{x}^T{\rm d}\vec{x}
	\; .
	\]
	In particular, $A \mapsto \gamma(AU)$ has a critical point at $I_n$ when restricted to determinant-one matrices if and only if $U$ is in isotropic Gaussian position.
\end{fact}
\begin{proof}
	Simply apply Claim~\ref{clm:grad_h} with $f(x) = e^{-\pi x}$ and notice that a differentiable function $h: \R^{n \times n} \to \R$ has a critical point at $I_n$ when restricted to the set of determinant-one matrices if and only if its gradient is proportional to the identity. This follows from the fact that the tangent space to the set of determinant-one matrices at $I_n$ is the space orthogonal to the identity. 
\end{proof}

We will also need the following result due to Cordero-Erausquin, Fradelizi, and Maurey~\cite{CFM04}, 
which is related to the so-called (B) conjecture due to Banaszczyk (see~\cite{Latala}). 

\begin{theorem}[\cite{CFM04}]
\label{thm:B_conjecture}
For any symmetric convex body $K \subset \R^n$, the function $\gamma(e^{M} K)$, where $M \in \R^{n\times n}$ ranges over all diagonal matrices, is log-concave.
\end{theorem}

\begin{proof}[Proof of Proposition~\ref{prop:maximum_mass}]
By scaling $K$, we may assume that $s = 1$. Let $A = UDV$ be the singular-value decomposition of $A$. (I.e., $D$ is a diagonal matrix with non-negative entries along the diagonal and $U$ and $V$ are orthogonal matrices.)  Since $D$ has determinant one, we may write $D  = e^{D'}$ for a diagonal matrix $D'$ with trace zero. 

Note that the Gaussian measure is invariant under orthogonal transformations, so that $\gamma(AK) = \gamma(UDVK) = \gamma(DVK)$. 
Let $K' := VK$, and note that $\gamma(K') = \gamma(K)$ and that
$K'$ is in isotropic Gaussian position,
since $V$ is an orthogonal transformation.

Let $\widehat{h}(M) := \gamma(e^M K')/|\det(e^M)|$. By Fact~\ref{fact:gaussian_gradient} and the chain rule, we have
\[
\grad_M \widehat{h}(M) |_{M = 0} = -2\pi \int_{K'} e^{-\pi \|\vec{x}\|^2}\vec{x} \vec{x}^T {\rm d}\vec{x} = -\alpha \cdot I_n
\; 
\]
for some scalar $\alpha \in \R$, where the second equality is simply the fact that $K'$ is in isotropic Gaussian position.
Let $X \subset \R^{n \times n}$ be the set of trace-zero diagonal matrices. Then, the function $\widehat{h}_X$ obtained by restricting $\widehat{h}$ to $X$ has a critical point at zero, since $\Tr( I_n M) = 0$ for any $M \in X$. By Theorem~\ref{thm:B_conjecture}, $\widehat{h}_X$ is log-concave, so that this critical point must be a global maximum. Therefore,
\[
\gamma(AK) = \gamma(DK') = \gamma(e^{D'} K') \leq \gamma(K') = \gamma(K)
\; ,
\] as needed.
\end{proof}

We now proceed to the second part of the proof of Theorem~\ref{thm:isotropic_mass}. Namely, we prove the following.

\begin{theorem}
	\label{thm:ell_position_mass}
	For any symmetric convex body $K \subset \R^n$ with volume one, there is a determinant-one matrix $A \in \mathrm{SL}_n(\R)$ such that
	$\gamma_{1/t}(A K) \geq 2/3$,
	where $t := 2\sqrt{3e}(\log_2 n + 2) < 10 (\log n + 2)$.
\end{theorem}

The proof is based on an important theorem that follows from the work of
Figiel and Tomczak-Jaegermann~\cite{FigielT79}, Lewis~\cite{Lewis79}, and Pisier~\cite{Pisier82}. 
We first need some definitions. 
Recall that any symmetric convex body $K \subset \R^n$ defines a norm $\|\cdot \|_K$ given by
\[
\|\vec{x}\|_K := \inf \{ s \ge 0 \ : \ \vec{x} \in sK \}
\; .
\]
We then define the $\ell$-norm on $\R^{n \times n}$ by
\[
\ell_K(A) := \Big(\int_{\R^n}  \|A\vec{x}\|_K^2 {\rm d} \gamma(\vec{x})\Big)^{1/2}
\; ,
\]
where ${\rm d}\gamma(\vec{x}) := e^{-\pi \|\vec{x}\|^2} {\rm d}\vec{x}$.
Finally, we recall that the \emph{polar body} of a symmetric convex body $K$ is given by
\[
K^\circ := \{ \vec{x} \in \R^n \ : \ \forall \vec{y} \in K,\ \inner{\vec{y}, \vec{x}} \leq 1 \}
\; ,
\]
which is itself a convex body. 

\begin{theorem}[{\cite{FigielT79,Lewis79,Pisier82}; see \cite[Theorem 4.4.3]{thesis/D12}}]
	\label{thm:ell_ellstar}
	For any symmetric convex body $K \subset \R^n$, there exists a determinant-one matrix $A \in \mathrm{SL}_n(\R)$ such that
	\[
	\ell_K(A) \ell_{K^\circ}((A^{T})^{-1}) \leq n(\log_2n + 2)/\pi 
	\; .
	\]
\end{theorem}

\begin{lemma}
	For any symmetric convex body $K \subset \R^n$ with volume one and any determinant-one matrix $A \in \mathrm{SL}_n(\R)$, we have 
	\[
	\ell_{K^\circ}((A^{T})^{-1}) \geq \sqrt{n/(2\pi)} \cdot r_n > n/(2\pi \sqrt{e})
	\; .
	\] 
	where $r_n := \vol(B_2^n)^{-1/n} > \sqrt{n/(2\pi e)}$ is the radius such that $\vol(r_n B_2^n) = 1$.
\end{lemma}
\begin{proof}
	Since $\ell_{K^\circ}((A^{T})^{-1}) = \ell_{A^T K^\circ}(I_n) = \ell_{(A^{-1} K)^\circ}(I_n)$, and using that $A^{-1} K$ is also a symmetric convex body of volume one, it suffices to prove the case $A=I_n$ in the statement. 
Unpacking the definitions, we see that 
	\begin{align}
	\ell_{K^\circ}(I_n)^2 &= \int_{\R^n} \|\vec{x}\|_{K^\circ}^2 {\rm d}\gamma(\vec{x}) \nonumber \\
	&= \int_{\R^n} \sup_{\vec{y} \in K}\inner{\vec{y}, \vec{x}}^2 {\rm d}\gamma(\vec{x}) \nonumber \\
	&= \Big(\int_{\R^n} \|\vec{x}\|^2 {\rm d}\gamma(\vec{x}) \Big) \cdot \Big( \int_{\R^n} \sup_{\vec{y} \in K}\frac{\inner{\vec{y}, \vec{x}}^2}{\|\vec{x}\|^2} {\rm d} \gamma(\vec{x}) \Big) \nonumber \\
	&= \frac{n}{2\pi} \cdot \int_{\R^n} \sup_{\vec{y} \in K}\frac{\inner{\vec{y}, \vec{x}}^2}{\|\vec{x}\|^2} {\rm d} \gamma(\vec{x})
	\; , \label{eq:lem_ell_step_1}
	\end{align}
	where
	in the third equality we use integration in polar coordinates.
	By Jensen's inequality, we have
	\begin{align}
	\int_{\R^n} \sup_{\vec{y} \in K}\frac{\inner{\vec{y}, \vec{x}}^2}{\|\vec{x}\|^2} {\rm d} \gamma(\vec{x}) \geq \Big(\int_{\R^n} \sup_{\vec{y} \in K} \frac{\inner{\vec{y}, \vec{x}}}{\|\vec{x}\|} {\rm d} \gamma(\vec{x})\Big)^2
	\; , \label{eq:lem_ell_step_2}
	\end{align}
	and by Urysohn's inequality (see~\cite[Theorem 1.5.11]{AGMBook}), we have
	\begin{align}
	\int_{\R^n} \sup_{\vec{y} \in K}\frac{\inner{\vec{y}, \vec{x}}}{\|\vec{x}\|} {\rm d} \gamma(\vec{x})  \geq r_n
	\; . \label{eq:lem_ell_step_3}
	\end{align}
	The result follows by combining  Eqs.~\eqref{eq:lem_ell_step_1}, \eqref{eq:lem_ell_step_2}, and~\eqref{eq:lem_ell_step_3}.
\end{proof}

The following corollary is an immediate consequence of the previous two results.

\begin{corollary}
	\label{cor:ell_position}
	For any symmetric convex body $K \subset \R^n$ with volume one, there exists a determinant-one matrix $A \in \mathrm{SL}_n(\R)$ such that $\ell_K(A) \leq  2\sqrt{e}(\log_2 n + 2)$.
\end{corollary}

\begin{proof}[Proof of Theorem~\ref{thm:ell_position_mass}]
	By Corollary~\ref{cor:ell_position}, there exists an $A \in \mathrm{SL}_n(\R)$ such that 
	\[
	\ell_K(A)^2 = \int_{\R^n}  \|A\vec{x}\|_K^2 {\rm d} \gamma(\vec{x}) \leq (2\sqrt{e}(\log_2 n + 2))^2
	\; .
	\]
	We will use $A^{-1}$ as our matrix. Note that 
	\[
	\gamma_{1/t}(A^{-1}K) = \int_{\R^n} 1_{\|\vec{x}\|_{A^{-1} K} \leq t} {\rm d}\gamma(\vec{x})  = \int_{\R^n} 1_{\|A\vec{x}\|_{K} \leq t} {\rm d}\gamma(\vec{x}) = 1 - \int_{\R^n} 1_{\|A\vec{x}\|_{K} > t} {\rm d}\gamma(\vec{x}) 
	\; .
	\] 
	The result then follows by Markov's inequality, which tells us that 
	\[
	 \int_{\R^n} 1_{\|A\vec{x}\|_{K} > t} {\rm d}\gamma(\vec{x})  \leq  \frac{1}{t^2} \cdot \int_{\R^n}  \|A\vec{x}\|_K^2 {\rm d} \gamma(\vec{x})  \leq  \frac{1}{t^2} \cdot (2 \sqrt{e}(\log_2 n + 2))^2 = \frac{1}{3}
	\; .
	\qedhere
	\]
\end{proof}

We now obtain Theorem~\ref{thm:isotropic_mass} as an immediate corollary of Proposition~\ref{prop:maximum_mass} and Theorem~\ref{thm:ell_position_mass}. 

\begin{proof}[Proof of Theorem~\ref{thm:isotropic_mass}]
	By Theorem~\ref{thm:ell_position_mass} and the monotonicity of $\gamma_s(AK)$ in $s$, there is some $A \in \mathrm{SL}_n(\R)$ such that $\gamma_s(A K)  \geq 2/3$. By Proposition~\ref{prop:maximum_mass}, $\gamma_s(K) \geq \gamma_s(A K) \geq 2/3$, as needed.
\end{proof}

\paragraph{Concentration of measure. } We will also need a standard lemma about the concentration of Gaussian measure. We first recall the Gaussian isoperimetric inequality~\cite{SC74,Borell75} (see also~\cite[Theroem 3.1.9]{AGMBook}).

\begin{theorem}[Gaussian isoperimetric inequality]
	\label{thm:Gaussian_isoperimetric}
	For any measurable set $U \subset \R^n$ and any $\tau > 0$,
	\[
	\gamma(U + \tau B_2^n) \geq \int_{-\infty}^{\sigma + \tau} e^{-\pi x^2}{\rm d} x
	\; ,
	\]
	where $\sigma \in [-\infty, \infty]$ is such that $\int_{-\infty}^\sigma e^{-\pi x^2}{\rm d}x = \gamma(U)$.
\end{theorem}

Recall that the \emph{inradius} of a convex body $K$ is defined as $\max \{ r \geq 0 \ : \ rB_2^n \subseteq K\}$,
i.e., the radius of the largest ball contained in the body.

\begin{lemma}
\label{lem:gamma_concentration}
If $K \subset \R^n$ is a convex body with $\gamma_{1/t}(K) \geq 2/3$ for some $t > 0$, then
\[
\gamma_{1/(t+\tau)}(K) \geq 1-e^{-\pi r^2\tau^2}/3
\; ,
\]
for any $\tau \geq 0$, where $r \geq 0$ is the inradius of $K$.
\end{lemma}
\begin{proof}
	Note that
	\[
	\gamma_{1/(t + \tau)}(K) = \gamma((t+ \tau)K) \geq \gamma(tK + r \tau B_2^n)
	\; ,
	\]
	since by definition $r B_2^n \subseteq K$. Applying Theorem~\ref{thm:Gaussian_isoperimetric}, we have
	\[
	\gamma_{1/(t + \tau)}(K)  \geq \int_{-\infty}^{\sigma + r\tau} e^{-\pi x^2}{\rm d}x 
	\; ,
	\]
	where $\sigma > 0$ is such that $\int_{-\infty}^{\sigma}e^{-\pi x^2}{\rm d}x  = 2/3$. Finally, we note that
	\[
	\int_{-\infty}^{\sigma + r\tau} e^{-\pi r^2}{\rm d}x = 1 - \int_{\sigma + r \tau}^{\infty }e^{-\pi x^2}{\rm d}x \geq 1 -e^{-\pi r^2 \tau^2} \int_{\sigma}^{\infty }e^{-\pi x^2}{\rm d}x = 1-e^{-\pi r^2 \tau^2}/3
	\; ,
	\]
	where the inequality follows from the fact that $\sigma$, $\tau$, and $r$ are non-negative, so that  
	$(x + r \tau)^2 \geq x^2 + r^2 \tau^2$ for all $x \ge \sigma$.
	The result follows.
\end{proof}

\subsection{Proof of Theorem~\ref{thm:RM}}
\label{sec:actualproofofmain}

We now use Theorem~\ref{thm:same_grad_VL} and Theorem~\ref{thm:isotropic_mass} to characterize local minima of $\gamma_s(\V(\lat))$.

\begin{theorem}
\label{thm:minima_have_mass}
If $\lat \subset \R^n$ corresponds to a local minimum (or maximum) of $\gamma_{1/t}(\V(\lat))$ over the set of determinant-one lattices, 
then $\V(\lat)$ is in isotropic Gaussian position with parameter $1/t$, and
\[
\gamma_{1/t}(\V(\lat)) \geq 2/3
\; ,
\]
where $t := 10(\log n + 2)$.
\end{theorem}
\begin{proof}
By Theorem~\ref{thm:same_grad_VL} with $f(x) = t^{n} \cdot e^{-\pi t^2 x}$, we have 
\begin{align*}
\grad_A \big( \gamma_{1/t}(\V(A \lat))/|\det(A)| \big) \big|_{A = I_n} 
&= 2\int_{\V(\lat)} f'(\|\vec{x}\|^2) \vec{x}\vec{x}^T{\rm d}\vec{x} \\
&= - 2\pi t^{n+2} \cdot \int_{\V(\lat)} e^{-\pi t^2 \|\vec{x}\|^2}\vec{x}\vec{x}^T{\rm d}\vec{x}
\; .
\end{align*}
Recall that $I_n$ corresponds to a local extremum of a differentiable function $g(A)$ restricted to the manifold of determinant-one matrices only if $\grad_A g(A) |_{A = I_n}$ is a scalar multiple of the identity. So, the above expression must be a multiple of the identity. I.e., $\V(\lat)$ is in isotropic Gaussian position. The result then follows from Theorem~\ref{thm:isotropic_mass}.
\end{proof}

Before moving to the proof of our main theorem, we need the following claim.

\begin{claim}
\label{clm:log_bound}
For any $x > 1$,
\[
e^{-2 \log^2 x} + e^{-2 \log^2(x/(x-1))} < 1 
\; .
\]
\end{claim}
\begin{proof}
	By symmetry, we may assume that $x \geq 2$. (Otherwise, we can replace $x$ with $x/(x-1)$.) 
	If $2 \leq x \leq 2.5$, then
	\[
	e^{-2 \log^2 x} + e^{-2 \log^2(x/(x-1))} < e^{-2 \log^2 2}  + e^{-2 \log^2(5/3)} < 1
	\; .
	\]
	A similar computation works if $2.5 \leq x \leq e$.
	Finally, using 
	the fact that $\log(x/(x-1)) = - \log(1 - 1/x) > 1/x$ for $x > 1$, we have 
	for any $x \ge e$ that
	\[
	e^{-2 \log^2 x} + e^{-2 \log^2(x/(x-1))} < \frac{1}{x^{2}} + e^{-2/x^2} < \frac{1}{x^{2}} + 1 - \frac{1}{x^2} = 1
	\; . \qedhere
	\]
\end{proof}

We now prove our main theorem in the special case when $\lat$ is a stable lattice. The full result will follow as a relatively straightforward corollary.

\begin{proposition}
\label{prop:RM_stable}
For any stable lattice $\lat \subset \R^n$, 
$
\rho_{1/t}(\lat) \leq \frac{3}{2}
$,
where $t := 10(\log n + 2)$.
\end{proposition}
\begin{proof}
By Lemma~\ref{lem:rho_gamma}, it suffices to show that $\gamma_{1/t}(\V(\lat)) \ge 2/3$. 
We assume for induction that
$\gamma_{1/(10(\log d + 2))}(\V(\lat')) \geq 2/3$ for any stable lattice $\lat'$ of rank $d < n$. (A quick check shows that this is true for $d = 1$.)
Since the set of stable lattices is compact by Item~\ref{item:compact} of Proposition~\ref{prop:stable_properties} 
and the function $\gamma_{1/t}(\V(\lat))$ is continuous, we may assume without loss of generality that $\lat$ corresponds to a global minimum of $\gamma_{1/t}(\V(\lat))$ over the set of stable lattices. If this global minimum is also a \emph{local} minimum over the set of determinant-one lattices, then by Theorem~\ref{thm:minima_have_mass}, $\gamma_{1/t}(\V(\lat)) \geq 2/3$, and we are done. 

Otherwise, $\lat$ lies on the boundary of the set of stable lattices. I.e., there is some primitive sublattice $\lat' \subset \lat$ of rank $d < n$  such that $\lat'$ and $\lat/\lat'$ are stable. (See Item~\ref{item:stable_boundary} of Proposition~\ref{prop:stable_properties}.) By 
Corollary~\ref{cor:monotone_Voronoi_wins} (applied to the non-decreasing function $x \mapsto -e^{-\pi (tx)^2}$) together with Claim~\ref{clm:fundamental_product}, we have 
\begin{align}
\gamma_{1/t}(\V(\lat)) \geq \gamma_{1/t}(\V(\lat/\lat' \oplus \lat')) = \gamma_{1/t}(\V(\lat/\lat')) \cdot \gamma_{1/t}(\V(\lat'))
\; .\label{eq:gammadirectsum}
\end{align}
Let $t_1 := 10(\log d+2)$ and $t_2 := 10(\log(n-d) + 2)$. By the induction hypothesis, we see that $\gamma_{1/t_1}(\V(\lat')) \geq 2/3$ and $\gamma_{1/t_2}(\V(\lat/\lat')) \geq 2/3$. By Lemma~\ref{lem:gamma_concentration}, we therefore have 
\[
\gamma_{1/t}(\V(\lat')) \geq 1-\frac{1}{3} \cdot e^{-2 \log^2(n/d)} \text{, and } 
\gamma_{1/t}(\V(\lat/\lat')) \geq 1-\frac{1}{3} \cdot e^{-2 \log^2(n/(n-d))}
\; ,
\]
where we have used the fact that the inradius of the Voronoi cell, which is equal to half the length of the shortest nonzero vector, is at least $1/2$ for stable lattices
(and the constant in the exponent is very loose).
Therefore, using~\eqref{eq:gammadirectsum},
\begin{align*}
\gamma_{1/t}(\V(\lat)) &\geq 
\Big(  1-\frac{1}{3} \cdot e^{-2 \log^2(n/d)}  \Big) \cdot \Big(1-\frac{1}{3} \cdot e^{-2 \log^2(n/(n-d))}\Big) \\
&> 1- \frac{1}{3} \cdot \big( e^{-2 \log^2(n/d)} + e^{-2 \log^2(n/(n-d))} \big) \\
&> \frac{2}{3}
\; ,
\end{align*}
where the last inequality follows from Claim~\ref{clm:log_bound} with $x := n/d$. 

So, for every stable lattice $\lat$, we have $\gamma_{1/t}(\V(\lat)) \geq 2/3$, and the result then follows from Lemma~\ref{lem:rho_gamma}.
\end{proof}

We now derive our main theorem as a corollary. 

\begin{proof}[Proof of Theorem~\ref{thm:RM}]
Let $\{\vec0\} = \lat_0 \subset \cdots \subset \lat_k = \lat$ be the canonical filtration of $\lat$. Recall from Item~\ref{item:stable_quotient} of Proposition~\ref{prop:stable_properties} that for $1 \le i \le k$, $\alpha_i \cdot (\lat_i/\lat_{i-1})$ is a stable lattice, where $\alpha_i := \det(\lat_i/\lat_{i-1})^{-1/\rank(\lat_i/\lat_{i-1})}$. 
Moreover, $\alpha_1 = \det(\lat_1)^{-1/\rank(\lat_1)} \le 1$ by our assumption on $\lat$. 
By Item~\ref{item:increasing_slopes} of Proposition~\ref{prop:stable_properties},
$\alpha_i$ is non-increasing with $i$, implying that 
$\alpha_i \le 1$ for $1 \le i \le k$. 
Therefore, by Lemma~\ref{lem:direct_sum_rho}, 
\[
\rho_{1/t}(\lat) \leq \rho_{1/t}\Big(\bigoplus_{i=1}^k \lat_i/\lat_{i-1} \Big) \leq \rho_{1/t}\Big(\bigoplus_{i=1}^k \alpha_i\cdot (\lat_i/\lat_{i-1}) \Big)
\; .
\]
By Item~\ref{item:direct_sum_stable} of Proposition~\ref{prop:stable_properties}, this direct sum of stable lattices is itself a stable lattice. The result then follows from Proposition~\ref{prop:RM_stable}.
\end{proof}

\section{Bounds on \texorpdfstring{$\rho_s(\lat)$}{rho} for all parameters and point-counting bounds}
\label{sec:count}

We first give the proof of Corollary~\ref{cor:counting}, which follows immediately from Theorem~\ref{thm:RM_all_parameters}.

\begin{proof}[Proof of Corollary~\ref{cor:counting}]
	For any $r > 0$
		\[
		|\lat \cap (rB_2^n + \vec{u})| \leq e^{\pi r^2/s^2} \rho_s(\lat - \vec{u}) \leq e^{\pi r^2/s^2} \rho_s(\lat)
		\; ,
		\]
		where the last inequality is Claim~\ref{clm:shifted_mass}.
		Item~\ref{item:RM_bound} then follows by plugging in $s = 1/t$ and applying Item~\ref{item:RM} of Theorem~\ref{thm:RM_all_parameters}.
		Item~\ref{item:convex_bound} follows by taking $s = r\sqrt{2\pi/n}$ and applying Item~\ref{item:RM_convex} of Theorem~\ref{thm:RM_all_parameters}. Finally, Item~\ref{item:dual_packing_bound} follows by taking $s = r\sqrt{2\pi/n}$ and applying Item~\ref{item:RM_dual} of Theorem~\ref{thm:RM_all_parameters}.
	\end{proof}

We now prove Theorem~\ref{thm:RM_all_parameters}, which gives bounds on the Gaussian mass for all parameters.
We start with Item~\ref{item:RM}, addressing parameters $s \leq 1/t$. 

\begin{theorem}[Slight strengthening of Item~\ref{item:RM} of Theorem~\ref{thm:RM_all_parameters}]
	For any lattice $\lat \subset \R^n$ with $\det(\lat') \geq 1$ for all sublattices $\lat' \subseteq \lat$, 
	\[
	\rho_{s}(\lat) \leq 1 + e^{-\pi \lambda_1(\lat)^2(1/s^2 - t^2)}/2 \leq 1 + e^{-\pi(1/s^2 - t^2)}/2
	\] for any $s \leq 1/t$, where $t := 10(\log n + 2)$.
\end{theorem}

\begin{proof}
	Note that for any $\vec{y} \in \lat \setminus \{\vec0\}$,
	\[
	\rho_s(\vec{y}) = \rho_{1/t}(\vec{y}) \cdot e^{-\pi \|\vec{y}\|^2(1/s^2 - t^2)} \leq  \rho_{1/t}(\vec{y})e^{-\pi \lambda_1(\lat)^2 (1/s^2 - t^2)} 
	\; .
	\]
	The result
	follows by summing over all $\vec{y}\in \lat \setminus \{\vec0\}$ and applying Theorem~\ref{thm:RM}.
	The second inequality uses the fact that $\lambda_1(\lat) \geq 1$.
\end{proof}

We now prove the ``high-parameter analogue'' of Theorem~\ref{thm:RM}. The proof uses Theorem~\ref{thm:RM} and duality. 

\begin{theorem}[Item~\ref{item:RM_dual} of Theorem~\ref{thm:RM_all_parameters}]
	\label{thm:dual_RM}
	For any lattice $\lat \subset \R^n$ with $\det(\lat') \geq 1$ for all sublattices $\lat' \subseteq \lat$ and any parameter $s \geq t$,
	$
	\rho_s(\lat) \leq 2s^n
	$
	where $t := 10(\log n + 2)$.
\end{theorem}

\begin{proof}
	Recall the Poisson Summation Formula applied to the Gaussian mass (Eq.~\eqref{eq:PSF}):
	\[
	\rho_s(\lat) = \frac{s^n}{\det(\lat)} \cdot \rho_{1/s}(\lat^*) 
	\; .
	\]
	Assume first that $\lat$ is stable. Then, by Theorem~\ref{thm:RM} and the fact that the dual of a stable lattice
	is stable (Item~\ref{item:dual_stable} of Proposition~\ref{prop:stable_properties}),
	\[
	\rho_s(\lat)  = s^n \cdot \rho_{1/s}(\lat^*) \leq s^n \cdot \rho_{1/t}(\lat^*)  \leq 2s^n
	\; .
	\]
	For a general lattice $\lat \subset \R^n$, let $\{\vec0\} = \lat_0 \subset \lat_1 \subset \cdots \subset \lat_k = \lat$ be the canonical filtration of $\lat$. Recall from Item~\ref{item:stable_quotient} that $\alpha_i \cdot (\lat_i/\lat_{i-1})$ is stable for some $\alpha_i$. Furthermore, by Item~\ref{item:increasing_slopes} of Proposition~\ref{prop:stable_properties}, $\alpha_i \leq 1$. Then, by Lemma~\ref{lem:direct_sum_rho},
	\[
	\rho_s(\lat) \leq \rho_s\Big(\bigoplus \lat_i/\lat_{i-1}\Big) \leq \rho_s\Big(\bigoplus \alpha_i \cdot \lat_i/\lat_{i-1}\Big) \leq 2s^n
	\; ,
	\]
	where the last inequality follows from the fact that the direct sum of stable lattices is stable 
	together with the bound proven above for stable lattices. (See Item~\ref{item:direct_sum_stable} of Proposition~\ref{prop:stable_properties}.)
\end{proof}

The rest of this section is dedicated to the proof of Item~\ref{item:RM_convex} of Theorem~\ref{thm:RM_all_parameters}. 
Note that we already have a bound on $\rho_{s}(\lat)$ for $s \leq 1/t$ and for $s \geq t$, but we currently have no non-trivial bound for intermediate parameters $1/t < s < t$. To remedy this, we show in Theorem~\ref{thm:log_convex} below that $\rho_{e^\sigma}(\lat)$ is ``approximately log-convex,'' which allows us to interpolate between these two bounds. 
In the proof of Theorem~\ref{thm:log_convex}, we are unable to work with $\rho_{e^\sigma}(\lat)$ directly,
so we instead show 
that it can be approximated by $\gamma_{e^\sigma}(\V(\lat))$ (Lemma~\ref{lem:rho_gamma_2}).
We then notice that the latter function is log-concave by Theorem~\ref{thm:B_conjecture}.

\begin{claim}
	\label{clm:shifted_Voronoi}
	For any lattice $\lat \subset \R^n$, $\vec{y} \in \lat$, and $s > 0$, 
	\[
	\rho_s(\vec{y}) \gamma_s(\V(\lat)) \leq \gamma_s(\V(\lat) + \vec{y}) \leq \gamma_s(\V(\lat))
	\]
\end{claim}
\begin{proof}
	By scaling appropriately, we may assume that $s = 1$.
	We have 
	\begin{align*}
	\gamma(\V(\lat) + \vec{y}) &= \int_{\V(\lat)} e^{-\pi \|\vec{x} + \vec{y}\|^2}{\rm d}\vec{x}\\
	&= \rho(\vec{y})\int_{\V(\lat)} \rho(\vec{x}) e^{-2 \pi \inner{\vec{y}, \vec{x}}}{\rm d}\vec{x}\\
	&= \rho(\vec{y})\int_{\V(\lat)} \rho(\vec{x}) \cosh(2\pi \inner{\vec{y}, \vec{x}}){\rm d}\vec{x}
	\; ,
	\end{align*}
	where we have used the symmetry of the Voronoi cell in the last line. The lower bound now follows from noting that $\cosh(2\pi \inner{\vec{x}, \vec{y}}) \geq 1$. For the upper bound, we recall that, by definition, any vector in the Voronoi cell $\vec{x} \in \V(\lat)$ satisfies $\inner{\vec{y}, \vec{x}} \leq \|\vec{y}\|^2/2$ for any lattice vector $\vec{y} \in \lat$. Therefore, $\cosh(2\pi \inner{\vec{y},\vec{x}}) \leq \cosh(\pi\|\vec{y}\|^2) \leq 1/\rho(\vec{y})$, as needed.
\end{proof}

\begin{lemma}
	\label{lem:rho_gamma_2}
	For any lattice $\lat \subset \R^n$ and any $s > 0$, 
	\[
	e^{-4n}/2 \leq \gamma_s(\V(\lat)) \rho_s(\lat) \leq 1
	\; .
	\]
\end{lemma}
\begin{proof}
	The upper bound is Lemma~\ref{lem:rho_gamma}, repeated for comparison.
	By scaling appropriately, we may assume that $s = 1$.
	Recall that $\int_{\R^n} \|\vec{x}\|^2 e^{-\pi \|\vec{x}\|^2}{\rm d}\vec{x} = n/(2\pi)$. It follows from Markov's inequality that $\int_{\sqrt{n/\pi} B_2^n} e^{-\pi \|\vec{x}\|^2} {\rm d} \vec{x} \geq 1/2$. Let
	\[
	Y := \{ \vec{y} \in \lat \ : \ (\V(\lat) + \vec{y} ) \cap \sqrt{n/\pi} B_2^n \neq \emptyset \}
	\; .
	\]
	I.e., $Y$ is the set of vectors $\vec{y} \in \lat$ such that there exists some $\vec{x} \in \sqrt{n/\pi} B_2^n$ with $\|\vec{y} - \vec{x}\| \leq \|\vec{y}' - \vec{x}\|$ for every $\vec{y}' \in \lat$. By taking $\vec{y}' = \vec0$, we immediately see that $Y \subseteq \lat \cap 2\sqrt{n/\pi} B_2^n$. Recalling that the Voronoi cell tiles space, we have
	\begin{align*}
	1/2
	&\leq \int_{\sqrt{n/\pi} B_2^n} e^{-\pi \|\vec{x}\|^2} {\rm d}\vec{x}\\
	&\leq \sum_{\vec{y} \in Y} \gamma(\V(\lat) + \vec{y}) \\
	&\leq |Y| \cdot \gamma(\V(\lat)) &\text{(Claim~\ref{clm:shifted_Voronoi})} \\
	&\leq |\lat \cap 2\sqrt{n/\pi } B_2^n| \cdot \gamma(\V(\lat))\\
	&\leq e^{4n} \rho(\lat) \gamma(\V(\lat))
	\; ,
	\end{align*}
	as needed.
\end{proof}

We now prove the ``approximate log-convexity'' of $\rho_{e^\sigma}(\lat)$.

\begin{theorem} 
	\label{thm:log_convex}
	For any lattice $\lat \subset \R^n$ and any $t_1 > s > t_2 > 0$,
	\[
	\rho_s(\lat) \leq 2e^{4n} \rho_{t_1}(\lat)^\tau \rho_{t_2}(\lat)^{1-\tau}
	\; ,
	\]
	where $\tau := \log(s/t_2)/\log(t_1/t_2)$. 
\end{theorem}
\begin{proof}
	We have
	\begin{align*}
		\rho_s(\lat) &\leq \frac{1}{\gamma_s(\V(\lat))} &\text{(Lemma~\ref{lem:rho_gamma})}\\
		&\leq \frac{1}{\gamma_{t_1}(\V(\lat))^\tau\gamma_{t_2}(\V(\lat))^{1-\tau}} &\text{(Theorem~\ref{thm:B_conjecture})}\\
		&\leq 2e^{4n} \rho_{t_1}(\lat)^\tau \rho_{t_2}(\lat)^{1-\tau} &\text{(Lemma~\ref{lem:rho_gamma_2})}
		\; ,
	\end{align*}
	as needed.
\end{proof}

\begin{corollary}[Item~\ref{item:RM_convex} of Theorem~\ref{thm:RM_all_parameters}]
		\label{cor:intermediate_parameters}
	For any lattice $\lat \subset \R^n$ with $\det(\lat') \geq 1$ for all $\lat' \subseteq \lat$ and any parameter $1/t < s < t$, we have
	\[
		\rho_s(\lat) \leq 4  (e^8 st)^{n/2}
	\; ,
	\]
	where $t := 10(\log n + 2)$.
\end{corollary}
\begin{proof}
	Let $\tau := (1-\log s/\log t)/2$. Then,
	\begin{align*}
		\rho_s(\lat)
		&\leq 2 e^{4n} \rho_{1/t}(\lat)^{\tau} \cdot \rho_{t}(\lat)^{1-\tau} &\text{(Theorem~\ref{thm:log_convex})}\\
		&\leq 2^{1+\tau} e^{4n} \rho_{t}(\lat)^{1-\tau} &\text{(Theorem~\ref{thm:RM})}\\
		&\leq 4 e^{4n} t^{(1-\tau)n} &\text{(Corollary~\ref{thm:dual_RM})}\\
		&= 4 (e^8st)^{n/2}
		\; ,
	\end{align*}
	as needed.
\end{proof}

\section{Proof of the covering radius approximation}
\label{sec:KL}

We will need the following lemma, which is implicit in~\cite{banaszczyk}. 

\begin{lemma}
	\label{lem:mu_eta} 
	For any lattice $\lat \subset \R^n$ and $t > 0$ such that $\rho_{1/t}(\lat^*) \leq 3/2$,
	\[
	\mu(\lat) < \Big(\sqrt{\frac{n}{2\pi}} + 1\Big) \cdot  t
	\; .
	\]
	\end{lemma}
	\begin{proof}
	By scaling the lattice, we may assume without loss of generality that $t = 1$. Let $r := ((1+4/\sqrt{n})/(2\pi))^{1/2}$. 
	By Lemma~\ref{lem:banaszczyk}, for any $\vec{t} \in \R^n$, we have
	\[
	\rho((\lat - \vec{t}) \setminus r\sqrt{n}B_2^n)  \leq \big(\sqrt{2\pi e r^2} e^{-\pi r^2}\big)^n \cdot \rho(\lat) = e^{-2\sqrt{n}} (1+4/\sqrt{n})^{n/2} \cdot \rho(\lat) < \rho(\lat)/3
	\; ,
	\]
	where the last inequality follows by noting that it holds for $n = 1$ and that $e^{-2x} \cdot (1+4/x)^{x^2/2}$ is a decreasing function in $x$ for $x > 0$ (a fact that can be proven using a tedious but straightforward computation). 
	On the other hand, it is an easy consequence of the Poisson Summation Formula (see, e.g.,~\cite{MR04}) that for any $\vec{t} \in \R^n$,
		\begin{equation*}
		\rho(\lat - \vec{t}) \geq \frac{2-\rho(\lat^*)}{\rho(\lat^*)} \cdot \rho(\lat) \geq \rho(\lat)/3
		\;.
		\end{equation*}
	Therefore, $(\lat -\vec{t}) \cap r\sqrt{n}B_2^n$ is nonempty, and in particular, $\dist(\vec{t},\lat) \leq r\sqrt{n} < \sqrt{n/(2\pi)} + 1$. Since this holds for arbitrary $\vec{t}$, we have $\mu(\lat) < \sqrt{n/(2\pi)} + 1$, as needed.
	\end{proof}

We now note that Theorem~\ref{thm:RM} (together with Lemma~\ref{lem:mu_eta}) immediately implies a bound on the covering radius of stable lattices.

\begin{theorem}
\label{thm:KL_stable}
For any stable lattice $\lat \subset \R^n$, 
\[
\mu(\lat) \leq 4\sqrt{n}(\log n + 10) 
\; .
\]
\end{theorem}
\begin{proof}
Let $t := 10(\log n + 2)$. 
Since $\lat^*$ is also stable (by Item~\ref{item:dual_stable} of Proposition~\ref{prop:stable_properties}), 
Theorem~\ref{thm:RM} implies that $\rho_{1/t}(\lat^*) \leq 3/2$. 
Applying Lemma~\ref{lem:mu_eta}, we have 
\[
\mu(\lat) \leq (\sqrt{n/(2\pi)} + 1)  \cdot t 
< 4\sqrt{n} (\log n + 10)
\; ,
\] 
as needed.
\end{proof}

Next, we show (Proposition~\ref{prop:KL_stable_to_unstable}) how to reduce the case of general lattices to the stable case.
We will need the following technical lemma, which is a slight modification of \cite[Lemma 4.9]{DR16} (with an essentially identical proof).

\begin{lemma}[Reverse AM-GM]
\label{lem:sorted_rev_AM-GM}
Let $0 < a_1 < \cdots < a_k$ and $d_1,\dots,d_k \in \N$,
and for $j=1,\ldots,k$, define $m_j := \sum_{i \geq j} d_i$. Then,
\begin{align*}
\sum_{i=1}^k d_i a_i 
&\leq 2e \cdot \ceil{\log(2m_1)} \cdot \max_j m_j \Big(\prod_{i \geq j}
a_i^{d_i} \Big)^{1/m_{j}}  
\; .
\end{align*}
\end{lemma}
\begin{proof}
For $\ell=1,2,\ldots$, let 
$S_\ell := \{j \ : \ e^{-\ell} a_k < a_{j} \leq e^{1-\ell} a_k \}$, and let $j_{\ell} := \min \{j \in S_\ell\}$. For non-empty $S_\ell$, we have
\[
\sum_{i \in S_\ell} d_i a_i \leq m_{j_\ell} e^{1-\ell} a_k \le em_{j_\ell} a_{j_\ell} \leq e m_{j_\ell} \cdot \Big(\prod_{i \geq j_\ell}
a_i^{d_i}\Big)^{1/m_{j_\ell} }
\; .
\]
Let $\ell^* := \ceil{\log(2m_1)}$. By the above inequality, it suffices to argue that there exists an $\ell$ such that $2\ell^* \cdot \sum_{S_\ell} d_i a_i  \geq  \sum_{i=1}^k d_i a_i$. Indeed,
\[
\sum_{\ell =1}^{\ell^*} \sum_{i \in S_\ell} d_i a_i = \sum_{i =1}^k d_i a_i - \sum_{\ell > \ell^*} \sum_{i \in S_\ell} d_i a_i >  \sum_{i =1}^k d_i a_i - m_1 \cdot \frac{a_k}{2m_1} \geq \frac{1}{2} \cdot \sum_{i =1}^k d_i a_i
\; ,
\]
where in the last inequality we have used that $d_k \ge 1$.
Therefore, there exists an $\ell$ such that
\[
\sum_{i \in S_\ell} d_i a_i \geq \frac{1}{2\ell^*} \cdot \sum_{i =1}^k d_i a_i
\; ,
\]
as needed.
\end{proof}

Recall that
\[
\mu_{\det}(\lat) := \max_{W \subset \R^n} \sqrt{\dim(W^\perp)} \cdot \det(\pi_{W^\perp}(\lat))^{\frac{1}{\dim(W^\perp)}}
\; ,
\]
where the maximum is over lattice subspaces $W \subset \R^n$ of $\lat$ (i.e., subspaces $W$ spanned by up to $n-1$ lattice vectors).

\begin{proposition}
	\label{prop:KL_stable_to_unstable}
	Let
	\[
	\KLCstable := \max_{d \leq n} \sup  \mu(\lat)/\sqrt{d}
	\; ,
	\]
	where the supremum is over stable lattices $\lat \subset \R^d$. Then, for any lattice $\lat \subset \R^n$, 
\[
\mu(\lat) \leq \sqrt{2e \ceil{\log(2n)}} \cdot \KLCstable \cdot \mu_{\det}(\lat)
\; .
\]
\end{proposition}
\begin{proof}
	Let $\{ \vec0\} = \lat_0 \subset \lat_1 \subset  \cdots \subset \lat_k = \lat$ be the canonical filtration of some lattice $\lat\subset \R^n$. 
	Let $d_i := \rank(\lat_i/\lat_{i-1})$.
	Note that $\lat_i/\lat_{i-1}$ is a scaling of a stable lattice, i.e., $\det(\lat_i/\lat_{i-1})^{-1/d_i} \cdot (\lat_i/\lat_{i-1})$ is stable. (See Item~\ref{item:stable_quotient} of Proposition~\ref{prop:stable_properties}.)
		 We therefore have by Claim~\ref{clm:fundamental_product} and Lemma~\ref{lem:Voronoi_wins} that
	\begin{align}
	\mu(\lat)^2 &\leq \mu\Big(\bigoplus_i \lat_i/\lat_{i-1}\Big)^2\nonumber \\
	&=\sum_{i} \mu(\lat_i/\lat_{i-1})^2 \nonumber \\
	&\leq \KLCstable^2 \cdot \sum_i d_i \det(\lat_i/\lat_{i-1})^{2/d_i} \label{eq:boundoncoveringamgm} \; .
	\end{align}
Next, we recall from Item~\ref{item:increasing_slopes} of Proposition~\ref{prop:stable_properties} that $a_i := \det(\lat_i/\lat_{i-1})^{2/d_i}$ is an increasing sequence, and we note that $\sum_{i \ge j} d_i = \rank(\lat/\lat_{j-1})$. We may therefore use Lemma~\ref{lem:sorted_rev_AM-GM} to bound Eq.~\eqref{eq:boundoncoveringamgm} from above by
\begin{align*}
&2e \ceil{\log(2n)}\cdot \KLCstable^2 \cdot   \max_{i} \rank(\lat/\lat_{i}) \cdot \det(\lat/\lat_{i})^{\frac{2}{\rank(\lat/\lat_{i})}}  \\
&\quad \leq 
2e \ceil{\log(2n)}\cdot \KLCstable^2 \max_{W \subset \R^n} \dim(W^\perp) \cdot \det(\pi_{W^\perp}(\lat))^{\frac{2}{\dim(W^\perp)}}
\; , 
\end{align*}
as needed.
\end{proof}

Theorem~\ref{thm:KL} now follows as an immediate corollary of the above results. In particular, we have $\KLCstable \leq 4(\log n + 10)$ and therefore $\sqrt{2e \ceil{\log(2n)}} \cdot \KLCstable \leq 10 (\log n + 10)^{3/2} $. The result then follows from Proposition~\ref{prop:KL_stable_to_unstable}.

\subsection{Connection with the Slicing Conjecture}
\label{sec:KL_slicing}

In this section, we prove Theorem~\ref{thm:KL_slicing}. 
The structure of the proof is based on the one suggested in~\cite{SW16},
as was the case for the proof of our main theorem in Section~\ref{sec:voronoi_mass}.

As in Section~\ref{sec:voronoi_mass}, we are unable to work with the lattice parameter $\mu(\lat)$ that interests us directly.\footnote{
	While~\cite{DSV12} give a characterization of lattices corresponding to local maxima of $\mu$, we are unable to obtain a sufficiently strong bound on the covering radius of these lattices. See~\cite{SW16} for more about this question.} Instead, 
we work with the lattice parameter
\[
\mubar(\lat) := \sqrt{ \frac{1}{\det(\lat)} \int_{\V(\lat) }\|\vec{x}\|^2 {\rm d}\vec{x}}
\; ,
\]
which gives a good approximation to $\mu$. The following tight result due to Magazinov~\cite{Magazinov} (and conjectured in~\cite{HLR09}) makes this precise. (See~\cite[Claim 3.1]{HLR09} for a slightly weaker result with a simple proof. See, e.g.,~\cite{ZF96,ConwaySloaneBook98,GMR05,HLR09,Magazinov} for more about $\mubar$.)

\begin{theorem}[\cite{Magazinov}]
	\label{thm:mu_mubar}
	For any lattice $\lat \subset \R^n$,
	\[
	\mubar(\lat)  \leq \mu(\lat) \leq \sqrt{3} \mubar(\lat)
	\; .
	\]
\end{theorem}

We now observe that Theorem~\ref{thm:same_grad_VL} is applicable to the function $\mubar(\lat)^2$. Recall that a symmetric convex body $K \subset \R^n$ is said to be \emph{isotropic} if $\int_K \vec{x} \vec{x}^T {\rm d}\vec{x} = \alpha \cdot I_n$ for some scalar $\alpha > 0$.

\begin{proposition}
\label{prop:mubar_local_max}
For any lattice $\lat \subset \R^n$, 
\[
\grad_A \mubar(A \lat)^2 |_{A = I_n} = \frac{2}{\det(\lat)} \int_{\V(\lat)}  \vec{x}\vec{x}^T {\rm d}\vec{x}
\; ,
\]
where $A \in \R^{n \times n}$ is a non-singular matrix.
In particular, if $\lat$ corresponds to a local maximum (or local minimum) of $\mubar(\lat)$ over the set of determinant-one lattices, then $\V(\lat)$ is isotropic.
\end{proposition}
\begin{proof}
To compute the gradient, we simply apply Theorem~\ref{thm:same_grad_VL} with $f(x) := x$, and recall that 
\[
\mubar(A \lat)^2 = \frac{1}{\det(\lat)} \cdot  \frac{1}{|\det(A)|}\int_{\V(A \lat)} f(\|\vec{x}\|^2){\rm d}\vec{x}
\; .
\]
The ``in particular'' follows from the fact that a differentiable function $g(A)$ restricted to the set of determinant-one matrices has a critical point at $A = I_n$ if and only if $\grad_A g(A) |_{A = I_n}$ is a scalar multiple of the identity.
\end{proof}

We define the (symmetric) isotropic constant  
\[
L_n^2 := \max_{d\leq n} \frac{1}{d} \cdot \sup_K  \int_{K}\|\vec{x}\|^2{\rm d}\vec{x}
\; ,
\]
where the supremum is taken over all isotropic symmetric convex bodies $K \subset \R^d$ of volume one.
 It is known to satisfy $1/(2\sqrt{3}) \leq L_n \leq n^{o(1)}$, and the Slicing Conjecture implies that $L_n$ is bounded by a universal constant~\cite{Bourgain91-slicing,Klartag06,Chen21,KlartagLehec2022}. (The lower bound is due to the hypercube, $[-1/2,1/2]^n$.) We note in passing that we are only concerned with the isotropic constant for Voronoi cells, which could conceivably be easier to bound than the isotropic constant for arbitrary convex bodies.

\begin{theorem}
\label{thm:KL_stable_slicing}
For any stable lattice $\lat \subset \R^n$, 
\[
\mu(\lat) \leq \sqrt{3}\mubar(\lat) \leq \sqrt{3n} L_n
\; .
\]
\end{theorem}
\begin{proof}
By Theorem~\ref{thm:mu_mubar}, it suffices to prove that $\mubar(\lat) \leq \sqrt{n} L_n$. 
Note that this is trivially true for $n = 1$.
	We assume for induction that $\mubar(\lat') \leq \sqrt{d} L_d \leq \sqrt{d} L_n$ for all stable lattices $\lat'$ of rank $d < n$. 
	Recall that the set of stable lattices is compact (Item~\ref{item:compact} of Proposition~\ref{prop:stable_properties}), so that we may assume without loss of generality that $\lat$ corresponds to a global maximum of the function $\mubar$ over this set. If this is also a \emph{local} maximum over the set of determinant-one lattices, then by Proposition~\ref{prop:mubar_local_max}, the Voronoi cell is isotropic, and we 
	have $\mubar(\lat) \leq \sqrt{n} L_n$ by the definition of $\mubar$ and $L_n$. Otherwise, $\lat$ must lie on the boundary of the set of stable lattices. 
	I.e., there is some primitive sublattice $\lat' \subset \lat$ of rank $0 < d < n$ such that $\lat'$ and $\lat/\lat'$ are both stable. (See Item~\ref{item:stable_boundary} of Proposition~\ref{prop:stable_properties}.) Applying the induction hypothesis and Corollary~\ref{cor:monotone_Voronoi_wins} (together with Claim~\ref{clm:fundamental_product}), we have
	\[
	\mubar(\lat)^2 \leq \mubar(\lat' \oplus \lat/\lat')^2 = \mubar(\lat')^2 + \mubar(\lat/\lat')^2 \leq dL_n^2 + (n-d)L_{n}^2 = nL_n^2
	\; ,
	\]
	as needed.
\end{proof}

As far as we know, it is entirely possible that $L_n = 1/(2\sqrt{3})$, i.e., that the hypercube $[-1/2,1/2]^n$ is the worst symmetric
body for the Slicing Conjecture. 
If this is true, then we get that for any stable lattice $\lat \subset \R^n$, $\mu(\lat) \leq \sqrt{n}/2$, which is tight for $\Z^n$.
Apart from being an interesting statement in its own right, 
it was shown by Shapira and Weiss~\cite{SW16} that such a result would imply 
the so-called Minkowski conjecture (see there for more information). 

We can now use Proposition~\ref{prop:KL_stable_to_unstable} to extend Theorem~\ref{thm:KL_stable_slicing} to all lattices $\lat \subset \R^n$.

\begin{theorem}
\label{thm:KL_slicing}
For any lattice $\lat \subset \R^n$, 
\[
\frac{1}{\sqrt{2\pi e}} \cdot \mu_{\det}(\lat) 
\leq 
\mu(\lat) 
\leq 
5 \sqrt{\log n + 1} \cdot L_n \cdot \mu_{\det}(\lat)
\; .
\]
\end{theorem}

As we observed in Footnote~\ref{foot:harmonic_cov_rad}, there are lattices
with $\mu(\lat) \geq C\sqrt{\log n} \cdot \mu_{\det}(\lat)$.
So, up to a constant factor, Theorem~\ref{thm:KL_slicing} gives the strongest possible upper bound on $\mu(\lat)$ in terms of $\mu_{\det}(\lat)$, assuming the Slicing Conjecture. 
We note that Dadush recently proved a variant of Theorem~\ref{thm:KL_slicing}~\cite[Theorem 2.5]{DadApproximatingCovering19}. He avoids the $\sqrt{\log n}$ factor loss by replacing $\mu_{\det}$ with a parameter that depends on 
determinants of multiple projections of $\lat$ simultaneously, rather than just one as in the definition of $\mu_{\det}$.
In particular, assuming the Slicing Conjecture, his result gives a characterization of the covering radius up to a constant factor in terms of determinants of projections.

\section{An optimal bound for extreme parameters}
\label{sec:all_params}
We now prove Theorem~\ref{thm:extreme_parameters}, which says that $\Z^n$ has maximal Gaussian mass amongst all lattices $\lat$ with $\det(\lat') \geq 1$ for all sublattices $\lat' \subseteq \lat$, for very small parameters $s \leq \sqrt{2\pi/(n+2)}$ \emph{and} for very large parameters $s \geq \sqrt{(n+2)/(2\pi)}$. 
The proof is similar to that of Theorem~\ref{thm:RM}, except here we work directly with $\rho_s(\lat)$ (instead of the proxy $\gamma_s(\V(\lat))$). Moreover, we show that $\rho_s(\lat)$ has 
\emph{no local maxima} for those values of $s$, which leads to a simpler proof and the clearly tight result. 
In order to show that local maxima do not exist, we will show that the Laplacian of $\rho_s(\lat)$ is always positive when $\lat$ is stable. 

In more detail, for a lattice $\lat$ and $s>0$ let $f_{\lat,s}:X \to \R$ be given by 
\[
f_{\lat, s}(A) := \rho_s(e^{A/2}\lat) = \sum_{\vec{y} \in \lat} e^{-\pi \vec{y}^T e^A \vec{y}/s^2}
\; ,
\]
where $X \subset \R^{n \times n}$ is the linear space of all symmetric matrices with zero trace.
Notice that as $A$ ranges over $X$,
$e^{A/2} := I_n + \sum_{i=1}^\infty (A/2)^i/i!$ ranges over all determinant-one 
positive-definite matrices. In particular, 
$e^{A/2}\lat$ ranges over all lattices of fixed determinant, up to orthogonal transformations.
(To see this, notice that any $\lat'$ of the same determinant as $\lat$ can be written as $\lat' = T \lat$ for some matrix $T$ of determinant one. Then, up to an orthogonal transformation, $\lat'$ equals $V^T D V \lat$ where $T = UDV$ is the singular value decomposition of $T$. Finally, notice that $V^T D V$ is a determinant-one positive-definite matrix.)
See~\cite[Section 1.1.3]{TerrasBook} for a more in-depth treatment of the space of determinant-one matrices.

Recall that the \emph{Laplacian} of a twice differentiable function $g:X \to \R$ is
given by
\[
\Delta_X g(A) := \sum_{i} \frac{\partial^2}{\partial E_i^2} g(A)
\; ,
\]
where the $E_i$ form an orthonormal basis of $X$
(under the inner product $\langle A,B \rangle = \Tr(A^T B)$), and 
\[
\frac{\partial^2}{\partial M^2} g(A) := \frac{\partial^2}{\partial r^2} g(A + rM) |_{r = 0}
\] 
is the directional second derivative of $g$ in the $M$ direction. 
One can show that the Laplacian does not depend on the choice of basis. 
Clearly, if the Laplacian is positive at $A$, then $A$ cannot correspond to a local maximum of $g$, since there must be at least one direction in which the second derivative is positive. 

The Laplacian of $f_{\lat ,s}$ is straightforward to calculate. 
It can be found, e.g., in the work by Sarnak and Str{\"o}mbergsson~\cite{SS06}
who used it to study local minima of $\rho_s(\lat)$.

\begin{claim}[{\cite[Eq. (46)]{SS06}}]
	\label{clm:Laplacian}
	Let $X \subset \R^{n \times n}$ be the space of trace-zero symmetric matrices. Then, for any lattice $\lat \subset \R^n$ and any parameter $s > 0$,
	\[
	\Delta_X f_{\lat,s}(0) = \frac{\pi}{s^2} \cdot \frac{n-1}{n} \cdot \sum_{\vec{y} \in \lat} \rho_s(\vec{y})  \|\vec{y}\|^2 \Big( \frac{\pi}{s^2} \cdot \|\vec{y}\|^2 -  \frac{n+2}{2} \Big)
	\; .
	\]
\end{claim}

\begin{proposition}
	\label{prop:no_local_maxima}
	For any lattice $\lat \subset \R^n$ and 
	\[
	0 < s \leq \sqrt{\frac{2\pi}{n+2}} \cdot \lambda_1(\lat)
	\; ,
	\] 
	$\lat$ cannot correspond to a local maximum of $\rho_s(\lat)$ over the set of determinant-one lattices.
	In particular, since stable lattices have $\lambda_1(\lat) \geq 1$, a stable lattice cannot correspond to a local maximum for $s \leq \sqrt{2\pi/(n+2)}$.
\end{proposition}
\begin{proof}
	It suffices to show that the Laplacian given in Claim~\ref{clm:Laplacian} is positive for such $\lat$. Indeed, the summand is zero for $\vec{y} = \vec0$, and since 
	\[
	\frac{\pi}{s^2} \cdot \lambda_1(\lat)^2 \geq \frac{n+2}{2}
	\; , 
	\]
	the summand is non-negative for all non-zero $\vec{y} \in \lat$. Finally, since any lattice contains vectors of arbitrarily large length, there must be some strictly positive terms in the sum. Therefore, the full sum is strictly positive, as needed.
\end{proof}

From this, we derive our main result for the special case of stable lattices.  

\begin{proposition}
	\label{prop:stable_small_s}
	For any $0 < s \leq \sqrt{2\pi/(n+2)}$ and stable lattice $\lat \subset \R^n$, $\rho_s(\lat) \leq \rho_s(\Z^n)$.
\end{proposition}
\begin{proof}
	Note that the result is trivial for $n = 1$. We assume for induction that the result holds for all dimensions less than $n$. Since the set of stable lattices is compact and $\rho_s(\lat)$ is a continuous function, we may assume that $\lat$ corresponds to a global maximum of $\rho_s(\lat)$ over the set of stable lattices. By Proposition~\ref{prop:no_local_maxima}, this cannot be a local maximum over the set of determinant-one lattices. So, $\lat$ must be on the boundary of the set of stable lattices. I.e., there is a non-trivial primitive sublattice $\lat' \subset \lat$ with $d := \rank(\lat')$ such that $\lat'$ and $\lat/\lat'$ are themselves stable lattices of rank strictly less than $n$. (See Item~\ref{item:stable_boundary} of Proposition~\ref{prop:stable_properties}.) Applying the induction hypothesis, we have by Lemma~\ref{lem:direct_sum_rho}
	that
	\[
	\rho_s(\lat) \leq \rho_s(\lat') \cdot \rho_s(\lat/\lat') \leq \rho_s(\Z^{d}) \cdot \rho_s(\Z^{n-d}) = \rho_s(\Z^n)
	\; ,
	\]
	where we have used the fact that $s \leq \sqrt{2\pi/(n+2)} \leq \min\{ \sqrt{2\pi/(d+2)},\ \sqrt{2\pi/(n-d+2)}\}$ in order to apply the induction hypothesis.
\end{proof}

We now ``invert the parameter'' using duality.

\begin{corollary}
	\label{cor:stable_big_s}
	For any $s \geq \sqrt{(n+2)/(2\pi)}$ and stable lattice $\lat \subset \R^n$, $\rho_s(\lat) \leq \rho_s(\Z^n)$.
\end{corollary}
\begin{proof}
	Recall that the dual $\lat^*$ of a stable lattice is itself stable. (See Item~\ref{item:dual_stable} of Proposition~\ref{prop:stable_properties}.) Furthermore, by the Poisson Summation Formula for the discrete Gaussian (Eq.~\eqref{eq:PSF}),
	\[
	\rho_s(\lat) = \frac{s^n}{\det(\lat)} \cdot \rho_{1/s}(\lat^*) \leq \frac{s^n}{\det(\lat)} \cdot \rho_{1/s}(\Z^n) = \rho_s(\Z^n)
	\; ,
	\]
	as needed,
	where the inequality follows from Proposition~\ref{prop:stable_small_s}, and the last equality follows from the Poisson Summation Formula applied to $\Z^n$.
\end{proof}

We can now prove Theorem~\ref{thm:extreme_parameters}.

\begin{proof}[Proof of Theorem~\ref{thm:extreme_parameters}]
	Let $\{\vec0\} = \lat_0 \subset \lat_1 \subset \cdots \subset \lat_k = \lat$ be the canonical filtration of $\lat$, and let $d_i := \rank(\lat_i/\lat_{i-1}) \leq n$. Then, by Lemma~\ref{lem:direct_sum_rho}, we have
	\[
	\rho_s(\lat) \leq \prod_i \rho_s(\lat_i/\lat_{i-1})  
	\;.
	\]
	Note that, if $s \leq \sqrt{2\pi/(n+2)}$, then we also have $s \leq \sqrt{2\pi/(d_i+2)}$ for all $i$. 
	And, $\alpha_i \cdot (\lat_i/\lat_{i-1})$ is a stable lattice for some $\alpha_i  \leq 1$.
	(See Items~\ref{item:stable_quotient} and~\ref{item:increasing_slopes} of Proposition~\ref{prop:stable_properties}.) So, in this case we may apply Proposition~\ref{prop:stable_small_s} to obtain
	\[
	\rho_s(\lat)  
	\leq \prod_i \rho_s(\alpha_i \cdot (\lat_i/\lat_{i-1})) \leq	  
	\prod_i \rho_s(\Z^{d_i}) = \rho_s(\Z^n)
	\; .
	\]
	If, on the other hand, $s \geq \sqrt{(n+2)/(2\pi)}$, then $s \geq \sqrt{(d_i+2)/(2\pi)}$ for all $i$, so we may similarly apply Corollary~\ref{cor:stable_big_s} to obtain the same result.
\end{proof}

\begin{remark}
It is possible to show that, in the setting of Theorem~\ref{thm:extreme_parameters}, $\rho_{s}(\lat) = \rho_{s}(\Z^n)$
if and only if $\lat$ is an orthogonal transformation of $\Z^n$. 
To see this, first notice that in order to get equality, 
all the $\alpha_i$ in the proof above must be one, i.e., $\lat$ must be stable. 
Next, we follow the induction argument in the proof of Proposition~\ref{prop:stable_small_s},
and recall the case of equality in Lemma~\ref{lem:direct_sum_rho}.
\end{remark}

\section{Tightness of our bounds}
\label{sec:examples}
In this section, we discuss the tightness of our bounds by considering some classes of lattices $\lat \subset \R^n$.

\subsection{Tightness of Item~\ref{item:RM_dual} of Theorem~\ref{thm:RM_all_parameters} for stable lattices}

It is an immediate consequence of the Poisson Summation Formula (Eq.~\eqref{eq:PSF}) that $\rho_s(\lat) \geq s^n/\det(\lat)$ for any $s > 0$ and $\lat \subset \R^n$. Combining this with Item~\ref{item:RM_dual} of Theorem~\ref{thm:RM_all_parameters}, we see that
\[
s^n \leq \rho_s(\lat) \leq 2s^n
\; 
\]
for any \emph{stable} lattice $\lat \subset \R^n$ and any $s \geq 10 (\log n + 2)$.
I.e., Item~\ref{item:RM_dual} of Theorem~\ref{thm:RM_all_parameters} is tight for all stable lattices up to a factor of two in the mass.

\subsection{The integer lattice \texorpdfstring{$\Z^n$}{}}

We first prove bounds on the Gaussian mass of $\Z^n$. In particular, the lower bound in Eq.~\eqref{eq:Zn_mass1} below shows that  $\rho_{\sqrt{\pi/\log n}}(\Z^n) \geq 3/2$, so that Theorem~\ref{thm:RM} is tight for $\Z^n$ up to a factor of $C\sqrt{\log n}$ in $t$. Similar bounds hold for Items~\ref{item:RM} and~\ref{item:RM_convex} of Theorem~\ref{thm:RM_all_parameters}.

\begin{claim}
	For any $n \geq 1$ and parameter $s > 0$,  
	\begin{equation}
		\label{eq:Zn_mass1}
	\big(1+ 2e^{-\pi /s^2} \big)^n \leq \rho_s(\Z^n) \leq  \big( 1 + (2+s) e^{-\pi/s^2}\big)^n
	\; ,
	\end{equation}
	and
	\begin{equation}
		\label{eq:Zn_mass2}
	s^n \cdot \big(1 + 2e^{-\pi s^2}\big)^n \leq \rho_s(\Z^n) \leq s^n \cdot \big(1 + (2+1/s) e^{-\pi s^2}  \big)^n
	\; .
	\end{equation}
\end{claim}
\begin{proof}
	Note that $\rho_s(\Z^n) = \rho_s(\Z)^n$. So, it suffices to bound $\rho_s(\Z)$.
	Furthermore, Eq.~\eqref{eq:Zn_mass2} follows from Eq.~\eqref{eq:Zn_mass1} and the Poisson Summation Formula (Eq.~\eqref{eq:PSF}). So, it suffices to prove Eq.~\eqref{eq:Zn_mass1} for the case $n = 1$.
	For the lower bound, we have
	\[
	\rho_s(\Z) = 1 + 2 \sum_{z = 1}^\infty e^{-\pi z^2/s^2} \geq 1 + 2 e^{-\pi /s^2}
	\; .
	\]
	For the upper bound, we write
	\[
	\rho_s(\Z) = 1 + 2e^{-\pi/s^2} + 2\sum_{z=2}^\infty e^{-\pi z^2/s^2} \leq 1 + 2e^{-\pi/s^2} + 2\int_{1}^\infty e^{-\pi x^2/s^2} {\rm d} x \leq 1 + (2+s) e^{-\pi/s^2}
	\; ,
	\]
	where we have used~\cite[Eq.~7.1.13]{AbramowitzS64} to bound the error function. 
\end{proof}

We now bound $|\Z^n \cap r B_2^n|$. 
Note that the lower bound in the next claim, which shows that $|\Z^n \cap r B_2^n| \geq e^{C r^2 \log(n/r^2)}$, is relatively close to the upper bound $ |\Z^n \cap r B_2^n| \leq e^{C' r^2 \log^2 n}$ given by Item~\ref{item:RM_bound} of Corollary~\ref{cor:counting}.
(We include a better upper bound on $|\Z^n \cap rB_2^n|$ below for completeness. See~\cite{NSDThesis} for a slightly tighter bound via a more careful application of the same proof and~\cite{MO90} for tighter bounds for $r = C\sqrt{n}$.)

\begin{claim}
	For any $n \geq 1$ and any radius $1 \leq r \leq \sqrt{n}$, 
	\[
	(2n/\floor{r^2})^{\floor{r^2}} \leq |\Z^n \cap r B_2^n| \leq (2e^3 n/\floor{r^2})^{\floor{r^2}}
	\; .
	\]
\end{claim}
\begin{proof}
	Since all points in $\Z^n$ have integer squared norm, we may assume without loss of generality that $r^2$ is an integer. For the lower bound, we note that the number of vectors of length $r$ whose coordinates lie in the set $\{-1,0,+1\}$ is
	\[
	2^{r^2} \binom{n}{r^2} \geq (2n/r^2)^{r^2}
	\; ,
	\] 
	as needed.
	
	For the upper bound, using Eq.~\eqref{eq:Zn_mass1} with $s := \sqrt{\pi/\log(2n/r^2)} < 4$,
	\begin{align*}
	|\Z^n \cap r B_2^n| 
	&\leq e^{\pi r^2/s^2} \rho_s(\Z^n) \\
	&\leq (2n/r^2)^{r^2} \cdot \Big(1+\frac{r^2(2+s)}{2n}\Big)^n \\
	&\leq  (2n/r^2)^{r^2} \cdot (1+3r^2/n)^n \\
	&\leq (2e^3n/r^2)^{r^2}
	\; ,
	\end{align*}
	as needed.
\end{proof}

\subsection{Random lattices}

There exists a unique probability measure $\mathscr{L}_n$ over the set of determinant-one lattices in $\R^n$ that is invariant under $\mathrm{SL}_n(\R)$~\cite{Siegel45}. (See, e.g.,~\cite{TerrasBook} or~\cite[Chapter 3]{Lekkerkerker_book}.) We call a random variable $\lat$ sampled from $\mathscr{L}_n$ a \emph{random lattice}, and we write this as $\lat \sim \mathscr{L}_n$. The purpose of this section is to prove the following result.

\begin{proposition}\label{prop:randomtight}
	For any sufficiently large $n$ and any $r \geq \sqrt{n} \log n$, 
	\[
	\Pr_{\lat \sim \mathscr{L}_n} \Big[ \text{$\lat$ is stable and } |\lat \cap r B_2^n| \geq \vol(rB_2^n)/2  \Big] \geq 1 - (Cn/r^2)^{n/2} - (C/n)^{n/2}
	\; ,
	\]
	where $C > 0$ is some universal constant.
	In particular, there exists a stable lattice $\lat$ satisfying 
	\begin{align}\label{eq:boundrandompoints}
	|\lat \cap r B_2^n| \geq \vol(rB_2^n)/2 
	= (4 \pi n)^{-1/2} (2\pi e r^2/n)^{n/2}(1+o(1))
	\;  ,
	\end{align}
	where the $o(1)$ term approaches zero as $n$ approaches $\infty$. 
\end{proposition}

Note that the lower bound in Eq.~\eqref{eq:boundrandompoints} is within a factor of $C \sqrt{n}$ of the upper bound
in Item~\ref{item:dual_packing_bound} of Corollary~\ref{cor:counting}, which applies to stable lattices. 

We will need the following three results.

\begin{theorem}[\cite{Siegel45}]
	\label{thm:siegal}
	For any $n \geq 2$ and any measurable set $S \subset \R^n$,
	\[
	\expect_{\lat \sim \mathscr{L}_n}[|(\lat \setminus \{\vec0\}) \cap S|] = \vol(S)
	\; .
	\]
\end{theorem}

\begin{theorem}[{\cite{Rogers55,schmidt60}; see \cite[Theorem 24.3]{Gruber_book}}]
	\label{thm:rogers}
	For $n \geq 3$ and any Borel set $S \subset \R^n$,
	\[
	\expect_{\lat \sim \mathscr{L}_n}\big[\big(|(\lat \setminus \{\vec0\}) \cap S| - \vol(S)\big)^2 \big] \leq C \vol(S)
	\; ,
	\]
	where $C > 0$ is some universal constant.
\end{theorem}

\begin{theorem}[\cite{ShapiraW14}]
	\label{thm:random_stable}
	For any sufficiently large $n$, an $n$-dimensional random lattice is stable with probability at least $1-(C/n)^{n/2}$, where $C > 0$ is some universal constant.
\end{theorem}

\begin{proof}[Proof of Proposition~\ref{prop:randomtight}]
	By Chebyshev's inequality, Theorem~\ref{thm:siegal}, and Theorem~\ref{thm:rogers}, there is some universal constant $C > 0$ such that
	\[
	\Pr_{\lat \sim \mathscr{L}_{n}}\big[|\lat \cap rB_2^n| < \vol(r B_2^n)/2 \big] \leq \frac{C}{\vol(r B_2^n)} \leq (C'n/r^2)^{n/2}
	\; .
	\]
	The result then follows by Theorem~\ref{thm:random_stable} and union bound.
\end{proof}

\newcommand{\etalchar}[1]{$^{#1}$}
\def\cprime{$'$}

\end{document}